\documentclass[onefignum,onetabnum]{siamart220329}

\usepackage{lipsum}
\usepackage{amsfonts}
\usepackage{amssymb}
\usepackage{graphicx}
\usepackage{epstopdf}
\usepackage{algpseudocode}
\usepackage{tensor}
\usepackage{verbatim}
\usepackage{mathtools}
\usepackage{cleveref}
\usepackage{framed}
\usepackage{tikz}
\usetikzlibrary{calc, shapes, angles, quotes, patterns, decorations.pathreplacing, cd}
\usepackage{subcaption}
\usepackage{bm}
\usepackage{esint}
\usepackage{dsfont}

\ifpdf
\DeclareGraphicsExtensions{.eps,.pdf,.png,.jpg}
\else
\DeclareGraphicsExtensions{.eps}
\fi

\newlength{\leftstackrelawd}
\newlength{\leftstackrelbwd}
\def\leftstackrel#1#2{\settowidth{\leftstackrelawd}%
	{${{}^{#1}}$}\settowidth{\leftstackrelbwd}{$#2$}%
	\addtolength{\leftstackrelawd}{-\leftstackrelbwd}%
	\leavevmode\ifthenelse{\lengthtest{\leftstackrelawd>0pt}}%
	{\kern-.5\leftstackrelawd}{}\mathrel{\mathop{#2}\limits^{#1}}}

\newcommand{\unitint}[0]{I}
\newcommand{\bdd}[1]{ \boldsymbol{#1} }
\newcommand{\unitvec}[1]{\bdd{#1}}

\newcommand{\vertiii}[1]{{\left\vert\kern-0.25ex\left\vert\kern-0.25ex\left\vert #1 
		\right\vert\kern-0.25ex\right\vert\kern-0.25ex\right\vert}}
\newcommand{\leftnorm}[2]{\tensor*[_L]{ \|#1\| }{_{#2}} }

\newcommand{\leftnormsup}[3]{\tensor*[_L]{ \|#1\| }{_{#2}^{#3} } }

\newcommand{\zznorm}[2]{\tensor*[_{00}]{ \|#1\| }{_{#2}} }

\newcommand{\zznormsup}[3]{\tensor*[_{00}]{ \|#1\| }{_{#2}^{#3} } }

\newsiamremark{remark}{Remark}
\newsiamremark{hypothesis}{Hypothesis}
\crefname{hypothesis}{Hypothesis}{Hypotheses}
\newsiamthm{claim}{Claim}

\headers{Stable Lifting of Polynomial Traces on Triangles}{C. Parker and E. S\"{u}li}

\title{Stable Lifting of Polynomial Traces on Triangles \thanks{Submitted to the editors DATE. \funding{The first author acknowledges that this material is based upon work supported by the National Science Foundation under Award No. DMS-2201487. }}}

\author{Charles Parker\thanks{ Mathematical Institute, University of Oxford, Andrew Wiles Building, Woodstock Road, Oxford OX2 6GG, UK  (\email{charles.parker@maths.ox.ac.uk}, \email{suli@maths.ox.ac.uk})} \and Endre S\"{u}li\footnotemark[2] }

\usepackage{amsopn}

\ifpdf
\hypersetup{
  pdftitle={Stable Lifting of Polynomial Traces on Triangles},
  pdfauthor={C. Parker and E. S\"{u}li}
}
\fi

\begin{document}

\maketitle

\begin{abstract}		
	We construct a right inverse of the trace operator $u \mapsto (u|_{\partial T}, \partial_n u|_{\partial T})$ on the reference triangle $T$ that maps suitable piecewise polynomial data on $\partial T$ into polynomials of the same degree and is bounded in all $W^{s, q}(T)$ norms with $1 < q <\infty$ and $ {s \geq 2}$. The analysis relies on new stability estimates for three classes of single edge operators. We then generalize the construction for $m$th-order normal derivatives, $m \in \mathbb{N}_0$.
\end{abstract}



\section{Introduction}
\label{sec:intro}

The lifting of polynomial traces defined on the boundary of a triangle $T$ to a function defined over the entire triangle $T$ plays an essential role in the numerical analysis of high order finite element and spectral element discretizations of partial differential equations (PDEs). One of the earliest and perhaps most widely used lifting operators was constructed by Babu\v{s}ka \& Suri \cite{BabSuri87} and later improved upon by Babu\v{s}ka et al. \cite{BCMP91}. The operator maps $H^{\frac{1}{2}}(\partial T)$ boundedly into $H^1(T)$ \textit{and} if the boundary datum is a continuous piecewise polynomial, then the lifting is also a polynomial of the same degree. In the context of second-order elliptic problems, this operator is used in the convergence analysis of the $hp$-finite element methods (FEM) to obtain optimal convergence rates e.g. \cite{BabSuri87,GuoBab10} and in the analysis of substructuring preconditioners e.g. \cite{AinCP21LEP,AinCP19StokesIII,BCMP91,SchMelPechZag08}. 3D analogues by Belgacem \cite{Belgacem} on the cube and Mu\`{n}oz-Sola \cite{Munoz97} on the tetrahedron have similarly been used in a priori error analysis. Some generalizations of the operator in \cite{BCMP91} with stability in $L^q(T)$ based Sobolev spaces were constructed in \cite{Melenk05} with applications to $hp$ quasi-interpolation operators. 

A plethora of other lifting operators have since been constructed. In the analysis of spectral element methods and polynomial inverse inequalities, extension operators bounded in weighted Sobolev spaces on squares and cubes play a key role; see e.g. \cite{Bern95,Bern07,Bern10,Bern89} and references therein. The lifting operators in \cite{Dem08,Dem09,Dem12} satisfy a commuting diagram property with the de Rham complex and arise in the analysis of high-order mixed methods for electromagnetic problems. More recently, $H^2(T)$-stable lifting operators were constructed in \cite{AinCP19Extension,Lederer18} and used to prove uniform $hp$ inf-sup stability for $H(\mathrm{div})$ elements \cite{Lederer18} and $H^1$ elements \cite{AinCP19StokesI} for Stokes flow, as well as optimal $H^2$ convergence rates for $C^1$ finite elements \cite{AinCP19StokesII}. The above list is by no means exhaustive, but demonstrates the ubiquity of polynomial lifting operators.

Currently available lifting operators are not sufficient for all applications. For example, the $p$-biharmonic equation, which appears in image denoising \cite{Houichet21}, and the stream function formulation of 2D incompressible flow of a power-law fluid \cite{DAlessio96} lead to nonlinear fourth-order PDEs posed in $W^{2, q}$. Consequently, a $W^{2, q}(T)$-stable polynomial lifting operator for the trace and normal derivative would be crucial for optimal a priori error estimates of conforming $C^1$ finite element discretizations. Additionally, the need for a polynomial lifting of the trace and normal derivative boundedly into $H^{3}(T)$ is encountered in the analysis of a high-order mixed finite element method for linear elasticity \cite{AznaranHuParker23}. Finally, the modeling of phase field crystal models \cite{Backofen07} and the evolution of a thin film \cite{Wise04}, among other applications, give rise to sixth-order PDEs. The analysis of $C^2$-conforming finite element methods would require an $H^3(T)$-stable lifting of a polynomial trace, normal derivative, and second-order normal derivative. 

The main contribution of this paper is the construction of lifting operators that are simultaneously stable in all appropriate $W^{s, q}(T)$ norms, lift compatible polynomial traces to polynomials, and apply to each of the applications above. 

We first consider the problem of lifting a trace $f$ and normal derivative $g$ into general $W^{s, q}(T)$ spaces, where $1 < q < \infty$ and the regularity ${s \geq 2}$ can be arbitrarily large. In particular, we construct a single operator $\tilde{\mathcal{L}}$ independent of $s$ and $q$ satisfying $\tilde{\mathcal{L}}(f, g)|_{\partial T} = f$, $\partial_n \tilde{\mathcal{L}}(f, g)|_{\partial T} = g$, and $\tilde{\mathcal{L}}$ is bounded from an appropriate boundary norm into $W^{s, q}(T)$. Additionally, if $f$ and $g$ are piecewise polynomials and satisfy certain compatibility conditions, then $\tilde{\mathcal{L}}(f, g)$ is a polynomial. We then construct lifting operators for the generalization of the above problem to $m$th-order normal derivative traces, $m \in \mathbb{N}_0$. The existence of a lifting operator satisfying the conditions in \cite{AinCP19Extension}, \cite{BCMP91}, and \cite{Melenk05}, respectively, follows from our results by taking $(m, s, q) = (1, 2, 2)$, $(m, s, q) = (0, 1, 2)$, and $(m, s, q) = (0, 1, q)$.

The remainder of the paper is organized as follows. In \cref{sec:trace-review}, we review the regularity of the trace $u|_{\partial T}$ and normal derivative $\partial_n u|_{\partial T}$ for a general $W^{s, q}(T)$ function, $1 < q <\infty$ and ${s \geq 2}$. We state in \cref{sec:main-result} the first main result concerning the existence of $\tilde{\mathcal{L}}$ satisfying the properties above. The construction of the operator $\tilde{\mathcal{L}}$, which consists of three families of single edge lifting operators detailed in \cref{sec:single-edge-defs}, is explicitly given in \cref{sec:constructing-lifting}. In \cref{sec:single-edge-stability}, we prove the continuity properties of the single edge operators. Finally, we generalize our construction to arbitrary order normal derivatives in \cref{sec:generalization}. 

\section{The first two traces of $W^{s, q}(T)$ functions}
\label{sec:trace-review}

	Let $T$ denote the reference triangle as depicted in \cref{fig:reference triangle}, and let $u \in W^{s, q}(T)$, where $1 < q < \infty$ and ${s \geq 1}$ are real numbers. In this section, we review the regularity properties of the trace $u|_{\partial T}$ and, when well-defined, the normal derivative $\partial_n u|_{\partial T}$, collecting results from \cite{Arn88,Grisvard85,Jerison95,Jonsson84}.

	We first define some notation. Given an open set $\mathcal{O} \subseteq \mathbb{R}^d$ with Lipschitz boundary, let $W^{k, q}(\mathcal{O})$, $k \in \mathbb{N}_0$, $q \in [1, \infty)$ denote the usual Sobolev spaces \cite{Adams03} equipped with the norm
	\begin{align*}
		\|v\|_{k, q, \mathcal{O}}^q := \sum_{ |\alpha| \leq k } \int_{\mathcal{O}} |D^{\alpha} v(x)|^q  \ dx,
	\end{align*}
	with the usual modification for $q = \infty$. We collect the $j$th-order derivatives into one $j$th-order tensor given by
	\begin{align*}
		(D^j u)_{i_1 i_2 \ldots i_j} = \partial_{i_1} \partial_{i_2} \cdots \partial_{i_j} u.
	\end{align*}
	For $k = 0$, $W^{0, q}(\mathcal{O}) = L^q(\mathcal{O})$, and we use the notation $\|\cdot\|_{q, \mathcal{O}}$ to denote the norm. For $k \in \mathbb{N}_0$ and real $\beta \in (0, 1)$ let $W^{k+\beta, q}(\mathcal{O})$ denote the standard fractional Sobolev-Slobodeckij space \cite{Adams03} with norm
	\begin{align*}
		\|v\|_{k+\beta, q, \mathcal{O}}^q := \|v\|_{k,q, \mathcal{O}}^q + \sum_{|\alpha| = k} \iint_{\mathcal{O} \times \mathcal{O}} \frac{ | D^{\alpha} v(x) -  D^{\alpha} v(y)|^q }{ |x-y|^{\beta q + d}  } \ dx \ dy.
	\end{align*}
	The space $W^{\beta, q}(\Gamma)$ for a $d-1$-dimensional subset $\Gamma \subseteq \partial \mathcal{O}$ is defined analogously (see e.g. \cite[\S 1.3.3]{Grisvard85}) with the norm
	\begin{align*}
		\|v\|_{\beta, q, \Gamma}^q := \|v\|_{q, \Gamma}^q +  \iint_{\Gamma \times \Gamma} \frac{ | v(x) -  v(y)|^q }{ |x-y|^{\beta q + d-1}  } \ dx \ dy.
	\end{align*}
	When $\Gamma$ is an edge of a polygon, we additionally define $W^{k+\beta, q}(\Gamma)$, $k \in \mathbb{N}$, as
	\begin{align*}
		W^{k+\beta, q}(\Gamma) := \{ w \in L^q(\Gamma) : \partial_t^j w \in L^q(\Gamma), \ j \in \{1, 2, \ldots, k\}, \text{ and }  \partial_t^k w \in W^{\beta, q}(\Gamma) \},
	\end{align*}
	where $\partial_t$ is the tangential derivative operator on $\Gamma$. The corresponding norm is then
	\begin{align*}
		\|u\|_{k+\beta, q, \Gamma}^q := \| u\|_{k-1, q, \Gamma}^q + \| \partial_t^k u\|_{\beta, q, \Gamma}^q. 
	\end{align*}
	
		\begin{figure}[htb]
		\centering	
		\begin{tikzpicture}[scale=0.7]
			\filldraw (0,0) circle (2pt) 
			node[align=center,below]{}
			-- (4,0) circle (2pt)
			node[align=center,below]{}	
			-- (0,4) circle (2pt) 
			node[align=center,above]{}
			-- (0,0);
			
			\draw (-0.75, 0) node[align=center,below]{$(0,0) = {\bdd{a}}_2$};
			\draw (5, 0) node[align=center,below]{$\bdd{a}_3 = (1, 0)$};
			\draw (0.875, 4) node[align=center,above]{$\bdd{a}_1 = (0, 1)$};
			
			\coordinate (a1) at (0,0);
			\coordinate (a2) at (4,0);
			\coordinate (a3) at (0,4);
			
			\coordinate (e113) at ($(a2)!1/3!(a3)$);
			\coordinate (e123) at ($(a2)!2/3!(a3)$);
			\coordinate (e1231) at ($(a2)!2/3+1/sqrt(32)!(a3)$);
			
			\coordinate (e213) at ($(a3)!1/3!(a1)$);
			\coordinate (e223) at ($(a3)!2/3!(a1)$);
			\coordinate (e2231) at ($(a3)!2/3+1/4!(a1)$);
			
			\coordinate (e313) at ($(a1)!1/3!(a2)$);
			\coordinate (e323) at ($(a1)!2/3!(a2)$);
			\coordinate (e3231) at ($(a1)!2/3+1/4!(a2)$);
			
			\draw ($(e113)+(0.2,0.2)$) node[align=center]{${\gamma}_2$};
			\draw ($(e213)+(0,0)$) node[align=center,left]{${\gamma}_3$};
			\draw ($(e313)+(0,0)$) node[align=center,below]{${\gamma}_1$};
			
			\draw[line width=2, -stealth] (e123) -- (e1231);
			\draw ($($(e123)!0.5!(e1231)$)+(-0.25,-0.25)$)
			node[align=center]{$\unitvec{t}_2$};
			\draw[line width=2, -stealth] (e123) -- ($(e123)!1!-90:(e1231)$);
			\draw ($($(e123)!0.5!-90:(e1231)$)+(0.25,-0.25)$)
			node[align=center]{$\unitvec{n}_2$};
			\draw[line width=2, -stealth] (e223) -- (e2231);
			\draw ($($(e223)!0.5!(e2231)$)+(0,0)$) 
			node[align=center, right]{$\unitvec{t}_3$};
			\draw[line width=2, -stealth] (e223) -- ($(e223)!1!-90:(e2231)$);
			\draw ($($(e223)!0.5!-90:(e2231)$)+(0,0)$) 
			node[align=center, above]{$\unitvec{n}_3$};
			\draw[line width=2, -stealth] (e323) -- (e3231);
			\draw ($($(e323)!0.5!(e3231)$)+(0,0)$) 
			node[align=center, above]{$\unitvec{t}_1$};
			\draw[line width=2, -stealth] (e323) -- ($(e323)!1!-90:(e3231)$);
			\draw ($($(e323)!0.5!-90:(e3231)$)+(0,0)$) 
			node[align=center, left]{$\unitvec{n}_1$};

			\draw (4/3, 4/3) node(T){${T}$};
		\end{tikzpicture}		
		\caption{Reference triangle $T$}
		\label{fig:reference triangle}
	\end{figure}
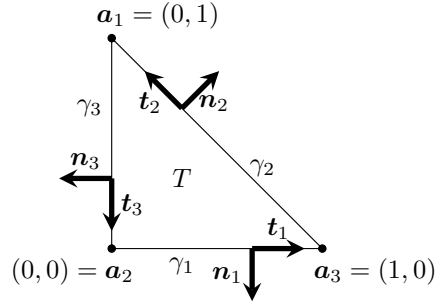
	
	We now return to $u \in W^{s, q}(T)$ with $1 < q < \infty$ and ${s \geq 1}$ and $T$ as in \cref{fig:reference triangle}. Since the boundary of $T$ is not smooth owing to the presence of the corners, the regularity of the trace of $u$ is limited. The primary tool to study its regularity is the standard $W^{\beta, q}(T)$ trace theorem (e.g. \cite[Theorem 3.1]{Jerison95} or \cite[p. 208 Theorem 1]{Jonsson84}): $W^{\beta, q}(T)$ embeds continuously into $W^{\beta-\frac{1}{q}, q}(\partial T)$ for $1/q < \beta < 1 + 1/q$. It will be useful to equip $W^{\beta-\frac{1}{q}, q}(\partial T)$ with the following equivalent norm (cf. \cite[Lemma 1.5.1.8]{Grisvard85} and \cite[p. 171-2]{Arn88}):
	\begin{align*}
		\| f \|_{\beta-\frac{1}{q}, q, \partial T}^q \approx_{\beta, q} \sum_{i=1}^{3}  \|f_i\|_{\beta-\frac{1}{q}, q, \gamma_i}^q + \begin{cases}
			\sum_{i=1}^{3} \mathcal{I}_i^q(f_{i+1}, f_{i+2}) & \text{if } \beta q = 2, \\
			0 & \text{otherwise},
		\end{cases}
	\end{align*}
	where $f_i$ denotes the restriction of $f$ to $\gamma_i$,  $\mathcal{I}_i^q(f, g)$ is defined by the rule
	\begin{align}
		\label{eq:vertex-integral-def}
		\mathcal{I}_i^q(f, g) := \int_{0}^{1} h^{-1}  |f(\bdd{a}_i - h \unitvec{t}_{i+1}) - g(\bdd{a}_i + h \unitvec{t}_{i+2})|^q  \ dh,
	\end{align}
	and indices are understood modulo 3. We use the standard notation $a \lesssim_c b$ to mean $a \leq C b$ where $C$ is a generic constant depending only on $c$, while $a \approx_c b$ means $a \lesssim_c b$ and $b \lesssim_c a$.
	
	Let $s = k + \beta$, where ${k \in \mathbb{N}}$ and $\beta \in [0, 1)$. The $j$th-order derivative tensor satisfies $D^{j} u \in W^{s-j, q}(T) \subset W^{1+\beta, q}(T)$, $0 \leq j \leq k-1$, and $D^k u \in W^{\beta, q}(T)$.  The trace theorem then gives
	\begin{align*}
	 \begin{cases}
			D^j u|_{\partial T} \in L^q(\partial T) & \text{for }  0 \leq j < s-\frac{1}{q}, \\
			D^{k-1} u|_{\partial T} \in W^{\beta+1-\frac{1}{q}, q}(\partial T) & {\text{if }  \beta q < 1}, \\
			D^{k} u|_{\partial T} \in W^{\beta-\frac{1}{q}, q}(\partial T) & \text{if } \beta q > 1.
		\end{cases}
	\end{align*}
	Note that in the final two cases above, the case $\beta q = 1$ is missing. In general, the trace of a $W^{1+\frac{1}{q}, q}(T)$ function does not have a globally defined tangential derivative in $L^q(\partial T)$ (see e.g. Proposition 3.2 \cite{Jerison95} and the subsequent discussion). Moreover, the trace of a $W^{k+\frac{1}{q}, q}(\mathbb{R}^2)$, $k \geq 1$, function on the real line $(-\infty, \infty) \times \{0\}$ belongs to a Besov space which cannot be identified with an integer-order Sobolev space unless $q=2$, in which case the trace belongs to $W^{k, 2}(\mathbb{R})$ (see e.g. \cite[Chapter 7]{Adams03} or \cite[p. 20 Theorem 4]{Jonsson84})\footnote{The case $\beta = 1/q$ and $q \neq 2$ is beyond the scope of this paper.}. Using standard arguments (cf. \cite[Theorem 6.1]{Arn88}), one can show that
	\begin{align}
		\label{eq:review:usual-trace-conditions}
		\begin{cases}
			\|D^j u \|_{q, \gamma_i} < \infty & \text{for } 0 \leq j < s-\frac{1}{q}, \\
			\|D^{k-1} u \|_{\beta + 1 - \frac{1}{q}, q, \gamma_i} < \infty & {\text{if } \beta q < 1}, \\
			\|D^{k} u \|_{\beta-\frac{1}{q}, q, \gamma_i} < \infty & {\text{if } \beta q > 1 \text{ or } (\beta, q) = (\frac{1}{2}, 2)}, \\
			\mathcal{I}_i^q(D^{k} u, D^{k} u) < \infty & \text{if } \beta q= 2. 
		\end{cases}
	\end{align} 
	Thanks to the Sobolev embedding theorem, we augment the above conditions with the following continuity condition: if $u \in W^{s, q}(T)$, then
	\begin{align}
		\label{eq:review:sobolev-embed-cont}
		D^j u \in C(\bar{T}) \quad \text{if } (s-j)q > 2, \ j \in \{0, 1, \ldots, k\}.
	\end{align}
	
	We first focus on the consequences of \cref{eq:review:usual-trace-conditions,eq:review:sobolev-embed-cont} for the trace $u|_{\partial T}$. We may express the $j$th-order tangential derivative of $u$ on $\gamma_l$ in terms of $D^j u$ as follows:
	\begin{align*}
		\partial_t^j u = (D^j u)_{i_1 i_2 \ldots i_j} (\unitvec{t}_l)_{i_1} (\unitvec{t}_l)_{i_2} \cdots (\unitvec{t}_l)_{i_j} \quad \text{on } \gamma_l.
	\end{align*}
	Thanks to \cref{eq:review:usual-trace-conditions}, $u|_{\gamma_i} \in W^{s-\frac{1}{q}, q}(\gamma_i)$, $i \in \{1,2,3\}$, whenever $(s, q) \in \mathcal{A}_0$, where
	\begin{align}
		\label{eq:admissible-set-def}
		\mathcal{A}_m := \left\{ (s, q) \in \mathbb{R}^2 : 1 < q < \infty, \ {s \geq m+1}, \text{ and } s - \frac{1}{q} \notin \mathbb{Z} \text{ if } q \neq 2 \right\}, \quad m \in \mathbb{N}_0.
	\end{align}
	Moreover, if $sq = 2$, then $\mathcal{I}_i^q(u, u) < \infty$, while if $sq > 2$, then \cref{eq:review:sobolev-embed-cont} shows that $u|_{\partial T}$ is continuous. In summary, the $0$th-order trace operator $\sigma^0$ defined by the rule
	\begin{align}
		\label{eq:sigma0-def}
		\sigma^0(f) := f \quad \text{on } \partial T
	\end{align}
	satisfies the following conditions for $f = u|_{\partial T}$ and $(s, q) \in \mathcal{A}_0$:
	\begin{enumerate}
		\item $W^{s-\frac{1}{q}, q}$ regularity on each edge:
		\begin{align}
			\label{eq:review:sigma0-edge-reg}
			\sigma_i^0(f) \in W^{s-\frac{1}{q}, q}(\gamma_i), \qquad i \in \{1,2,3\}.
		\end{align}

		\item Continuity at vertices: For $i \in \{1,2,3\}$, there holds
		\begin{subequations}
			\label{eq:review:sigma0-cont-combo}
				\begin{alignat}{2}
				\label{eq:review:sigma0-cont-1}
				\sigma_{i+1}^0(f)(\bdd{a}_{i}) =  \sigma_{i+2}^0(f)(\bdd{a}_{i}) & \qquad & &\text{if } sq > 2, \\
				\label{eq:review:sigma0-cont-2}
				\mathcal{I}_{i}^{q}(\sigma_{i+1}^0(f), \sigma_{i+2}^0(f)) < \infty& \qquad & & \text{if } sq = 2.
			\end{alignat}
		\end{subequations}
	\end{enumerate}
	
	We now turn the normal derivative $\partial_n u|_{\partial T}$ for $(s, q) \in \mathcal{A}_1$. Following the same arguments as above, we have
	\begin{align*}
		\partial_t^{j} \partial_n  u = (D^{j+1} u)_{i_1 i_2 \ldots, i_{j+1}}  (\unitvec{t}_l)_{i_1} (\unitvec{t}_l)_{i_2} \cdots (\unitvec{t}_l)_{i_j} (\unitvec{n}_l)_{i_{j+1}} \quad \text{on } \gamma_l,
	\end{align*}
	and so $\partial_n u|_{\gamma_l} \in W^{s-1-\frac{1}{q}, q}(\gamma_i)$, $i \in \{1,2,3\}$. However, $\partial_n u$ does not in general have any additional regularity owing to the jumps in the normal vector along $\partial T$. Instead, we turn to the operator $\sigma^1$ defined by the rule
	\begin{align}
		\label{eq:sigma1-def}
		\sigma^1(f, g) := \partial_t \sigma^0(f) \unitvec{t} + g \unitvec{n} = (\partial_t f) \unitvec{t} + g \unitvec{n}  \qquad \text{on } \partial T.
	\end{align} 
	Then, $\sigma^1(u, \partial_n u) = (\partial_t u) \unitvec{t} + (\partial_n u) \unitvec{n} = D u$ on $\partial T$, and so applying the edge regularity \cref{eq:review:usual-trace-conditions} to $\sigma^1$ gives \cref{eq:review:sigma1-edge-reg}. In particular, we recover $\partial_n u|_{\gamma_i} \in W^{s-1-\frac{1}{q}, q}(\gamma_i)$ via the relation $\partial_n u|_{\gamma_i} = \sigma_i^1(u, \partial_n u) \cdot \unitvec{n}_i$. However, we obtain additional conditions: The continuity condition \cref{eq:review:sobolev-embed-cont} gives \cref{eq:review:sigma1-cont}, while the integral condition \cref{eq:review:sigma1-cont-int} follows from \cref{eq:review:usual-trace-conditions}. Furthermore, if $(s-2)q > 2$, then $u \in C^2(\bar{\Omega})$. In particular, the mixed derivative $\partial_{t_{i+1} t_{i+2}} u$ is continuous at each vertex $\bdd{a}_i$, $i \in \{1,2,3\}$, which may be expressed in terms of $\sigma^1(u, \partial_n u)$ as follows:
	\begin{align*}
		\partial_t \sigma_{i+1}^1(u, \partial_n u)(\bdd{a}_i) \cdot \unitvec{t}_{i+2} &= \partial_{t_{i+1}} D u(\bdd{a}_i) \cdot \unitvec{t}_{i+2} = \partial_{t_{i+1} t_{i+2}} u(\bdd{a}_i) \\
		&= \partial_{t_{i+2}} D u(\bdd{a}_i) \cdot \unitvec{t}_{i+1} 
		= \partial_t \sigma_{i+2}^1(u, \partial_n u)(\bdd{a}_i) \cdot \unitvec{t}_{i+1}.
	\end{align*}
	Consequently, we obtain that the additional condition \cref{eq:review:sigma1-deriv} follows from \cref{eq:review:usual-trace-conditions,eq:review:sobolev-embed-cont}. In summary, the traces $f = u|_{\partial T}$ and $g = \partial_n u|_{\partial T}$ satisfy the following for all $(s, q) \in \mathcal{A}_1$:
	\begin{enumerate}
				
		\item $W^{s-1-\frac{1}{q}, q}$ regularity on each edge:
		\begin{align}
			\label{eq:review:sigma1-edge-reg}
			\sigma_i^1(f, g) \in W^{s-1-\frac{1}{q}, q}(\gamma_i) \qquad i \in \{1,2,3\}.
		\end{align}
		
		\item Continuity at vertices: For $i \in \{1,2,3\}$, there holds
		\begin{subequations}
			\label{eq:review:sigma1-cont-combo}
			\begin{alignat}{2}
				\label{eq:review:sigma1-cont}
				\sigma_{i+1}^1(f, g)(\bdd{a}_{i}) =  \sigma_{i+2}^1(f, g)(\bdd{a}_{i}) & \qquad & &\text{if } (s-1)q > 2, \\
				\label{eq:review:sigma1-cont-int}
				\mathcal{I}_{i}^{q}(\sigma_{i+1}^1(f, g), \sigma_{i+2}^1(f, g)) < \infty& \qquad & & \text{if } (s-1)q = 2.
			\end{alignat}
		\end{subequations}
		
		\item Higher derivative continuity at vertices: For $i \in \{1,2,3\}$, there holds
		\begin{subequations}
			\label{eq:review:sigma1-deriv}
			\begin{alignat}{2}
				\label{eq:review:sigma1-deriv-cont}
				\unitvec{t}_{i+2} \cdot \partial_t \sigma_{i+1}^1(f, g)(\bdd{a}_{i})  = \unitvec{t}_{i+1} \cdot \partial_t \sigma_{i+2}^1(f, g)(\bdd{a}_{i}) & \qquad & &\text{if } (s-2)q > 2, \\
				\label{eq:review:sigma1-deriv-cont-int}
				\mathcal{I}_{i}^{q}(\unitvec{t}_{i+2} \cdot \partial_t \sigma_{i+1}^1(f, g) , \unitvec{t}_{i+1} \cdot \partial_t \sigma_{i+2}^1(f, g)) < \infty& \qquad & & \text{if } (s-2)q = 2.
			\end{alignat}
		\end{subequations}
	\end{enumerate}
	 
	 Motivated by the above conditions, we define the space $X^{s,q}(\partial T)$, for $(s, q) \in \mathcal{A}_1$ as follows:
	 \begin{align}
	 	\label{eq:trace-1-space-def}
	 	X^{s,q}(\partial T) &:= \{ (f, g) \in L^q(T)^2 : \text{$(f, g)$ satisfy \cref{eq:review:sigma0-edge-reg}, \cref{eq:review:sigma0-cont-1}, and \cref{eq:review:sigma1-edge-reg,eq:review:sigma1-cont-combo,eq:review:sigma1-deriv}} \},
	 \end{align}
 	equipped with the norm
 	\begin{multline*}
	 	\| (f, g) \|_{X^{s,q}, \partial T}^q := \sum_{i=1}^{3} \left\{  \| f_i \|_{s-\frac{1}{q}, q, \gamma_i}^q + \| g_i \|_{s-1-\frac{1}{q}, q, \gamma_i}^q \right\} \\
	 	 + \sum_{i=1}^{3}  \begin{cases}
	 		 \mathcal{I}_i^q( \sigma_{i+1}(f, g), \sigma_{i+2}(f, g) ) & \text{if } (s-1)q = 2, \\
	 		\mathcal{I}_i^q( \unitvec{t}_{i+2} \cdot \partial_t \sigma_{i+1}(f, g), \unitvec{t}_{i+1} \cdot \partial_t \sigma_{i+2}(f, g)   ) & \text{if } (s-2)q = 2, \\
	 		0 & \text{otherwise}.
	 	\end{cases}
 	\end{multline*}
 	The preceding discussion then shows that for $u \in W^{s, q}(T)$, $(u|_{\partial T}, \partial_n u|_{\partial T}) \in X^{s, q}(\partial T)$. The following result shows that the converse is also true \cite[Theorem 6.1]{Arn88}.
 	\begin{theorem}
 		For every $(s, q) \in \mathcal{A}_1$, and $u \in W^{s, q}(T)$, there holds
 		\begin{align}
 			\label{eq:trace-thm}
 			(u|_{\partial T}, \partial_n u|_{\partial T}) \in X^{s,q}(\partial T)  \quad \text{with} \quad \|(u, \partial_n u)\|_{X^{s,q}, \partial T} \lesssim_{s, q} \|u\|_{s, q, T}.
 		\end{align}
 		Moreover, there exists a single linear operator $\mathcal{L} : \bigcup_{ (s, q) \in \mathcal{A}_1 } X^{s,q}(\partial T) \to W^{1, 1}(T)$ satisfying the following properties: For all $(s, q) \in \mathcal{A}_1$ and $(f, g) \in X^{s,q}(\partial T)$, $\mathcal{L}(f, g) \in W^{s,q}(T)$ and there holds
 		 \begin{align}
 		 	\label{eq:trace-inverse-prop}
 		 	\mathcal{L} (f, g)|_{\partial T} = f, \quad \partial_n \mathcal{L} (f, g)|_{\partial T} = g, \quad \text{and} \quad \|\mathcal{L} (f, g)\|_{s, q, T} \lesssim_{s, q} \|(f, g)\|_{X^{s,q}, \partial T}.
 		 \end{align}
	\end{theorem}
	In other words, there exists a single lifting operator of the trace and normal derivative that is stable from $X^{s,q}(\partial T)$ to $W^{s, q}(T)$ for all $(s, q) \in \mathcal{A}_1$.

	\section{Statement of the first main result}
	\label{sec:main-result}
	
	The present work constructs another lifting operator $\tilde{\mathcal{L}}$ satisfying the same interpolation and continuity properties as $\mathcal{L}$ of  \cref{eq:trace-inverse-prop}, with the additional property that if $(f, g) \in X^{s, q}(\partial T)$ are suitable piecewise polynomials of degree $p$ and $p-1$, then $\tilde{\mathcal{L}}(f, g)$ is a degree $p$ polynomial. The operators in \cite{Arn88,Grisvard85} do not satisfy this property as, among other reasons, they are constructed by using partition of unity on the boundary $\partial T$. Instead, we seek an alternative construction.
	
	The first issue at hand is to identify the appropriate conditions on $f$ and $g$ that ensure that a polynomial lifting exists. Let $\mathcal{P}_{r}(\mathcal{O})$ denote the set of polynomials of total degree at most $r \geq 0$ on an open set $\mathcal{O}$; for $r < 0$, set $\mathcal{P}_{r}(\mathcal{O}) = \{ 0 \}$. If the lifting of $f$ and $g$ is polynomial $\tilde{\mathcal{L}}(f, g) \in \mathcal{P}_{p}(T)$, then $\tilde{\mathcal{L}}(f, g) \in W^{s, q}(T)$ for all $(s, q) \in \mathcal{A}_1$, and so a necessary condition is that $(f, g)$ satisfy \cref{eq:review:sigma0-cont-1,eq:review:sigma1-cont,eq:review:sigma1-deriv-cont}. The following lemma shows that these conditions are also sufficient for $(f, g) \in X^{s, q}(\partial T)$.
	\begin{lemma}
		Let $f, g : \partial T \to \mathbb{R}$ with $f_i \in \mathcal{P}_p(\gamma_i)$ and $g_i \in \mathcal{P}_{p-1}(\gamma_i)$, $i \in \{1,2,3\}$, for some $p \in \mathbb{N}_0$. Then,  $(f, g) \in X^{s,q}(\partial T)$ for all $(s, q) \in \mathcal{A}_1$ if and only if $(f, g)$ satisfy \cref{eq:review:sigma0-cont-1,eq:review:sigma1-cont,eq:review:sigma1-deriv-cont}.
	\end{lemma}
	\begin{proof}
		Let $(f, g)$ be as in the statement of the lemma and assume that  $(f, g)$ satisfy \cref{eq:review:sigma0-cont-1,eq:review:sigma1-cont,eq:review:sigma1-deriv-cont}. Since polynomials are smooth, $(f, g)$ satisfy \cref{eq:review:sigma0-edge-reg,eq:review:sigma1-edge-reg}, while \cref{eq:review:sigma0-cont-2,eq:review:sigma1-cont-int,eq:review:sigma1-deriv-cont-int} follow from \cref{eq:review:sigma0-cont-1,eq:review:sigma1-cont,eq:review:sigma1-deriv-cont}. Thus, $(f, g) \in X^{s,q}(\partial T)$ for all $(s, q) \in \mathcal{A}_1$. The reverse implication follows by definition.
	\end{proof}
	
	We now state our first main result.
	\begin{theorem}
		\label{thm:tilde-l1-existence}
		There exists a single linear operator
		\begin{align*}
			\tilde{\mathcal{L}} : \bigcup_{ (s, q) \in \mathcal{A}_1 } X^{s,q}(\partial T) \to W^{1, 1}(T)
		\end{align*}
		satisfying the following properties: For all $(s, q) \in \mathcal{A}_1$ and $(f, g) \in X^{s,q}(\partial T)$, $\tilde{\mathcal{L}}(f, g) \in W^{s,q}(T)$ and there holds
		\begin{align}
			\label{eq:tilde-l1-prop}
			\tilde{\mathcal{L}} (f, g)|_{\partial T} = f, \quad \partial_n \tilde{\mathcal{L}} (f, g)|_{\partial T} = g, \quad \text{and} \quad \|\tilde{\mathcal{L}} (f, g)\|_{s, q, T} \lesssim_{s, q} \|(f, g)\|_{X^{s,q}, \partial T}.
		\end{align}
		Moreover, if $f_i \in \mathcal{P}_p(\gamma_i)$, $g_i \in \mathcal{P}_{p-1}(\gamma_i)$, $i \in \{1,2,3\}$, for some $p \in \mathbb{N}_0$, and satisfy \cref{eq:review:sigma0-cont-1,eq:review:sigma1-cont,eq:review:sigma1-deriv-cont}, then $\tilde{\mathcal{L}}(f, g) \in \mathcal{P}_{p}(T)$ and \cref{eq:tilde-l1-prop} holds for all $(s, q) \in \mathcal{A}_1$.
	\end{theorem}
	
	\section{Fundamental single edge operators}
	\label{sec:single-edge-defs}
	
	The construction of the operator $\tilde{\mathcal{L}}$ relies on three families of fundamental operators that lift a function defined on the unit interval $\unitint := (0, 1)$ to the reference triangle $T$. The first family is based on a convolution operator (see e.g. \cite[eq. (4.2)]{Arn88}, \cite{BCMP91}, \cite{Bern95}, \cite[p. 56, eq. (2.1)]{Bern07}, \cite[\S 2.5.5]{Necas11}): Given a nonnegative integer $m \in \mathbb{N}_0$, a smooth compactly supported function $b \in C_c^{\infty}(\unitint)$, and function $f : \unitint \to \mathbb{R}$, we define the operator $\mathcal{E}_m^{[1]}$ formally by the rule
	\begin{align*}
		\mathcal{E}_m^{[1]}(f)(x, y) := \frac{(-y)^m}{m!} \int_\unitint b(t) f(x + ty) \ dt, \qquad (x, y) \in T.
	\end{align*}
	We will use the notation $\mathcal{E}_m^{[1]}[b]$ when we want to make the dependence on $b$ explicit. Identifying $\gamma_1$  with $\unitint$ via the mapping 
	\begin{align}
		\label{eq:edge1-mapping}
		\varphi_1(h) := (1-h) \bdd{a}_2 + h \bdd{a}_3, \qquad h \in \unitint,
	\end{align}
	we use the notation $\mathcal{E}_{m}^{[1]}(f) := \mathcal{E}_m^{[1]}(f \circ \varphi_1)$ for $f : \gamma_1 \to \mathbb{R}$. 
	
	Analogous operators for edges $\gamma_2$ and $\gamma_3$ may be defined by mapping the triangle $T$ onto itself. More specifically, the map $R(x, y) = (1-x-y, x)^T$, takes $T \to T$ by rotating the labels of the vertices and edges in \cref{fig:reference triangle} counter-clockwise, while its inverse $R^{-1}(x, y) = (y, 1-x-y)^T$ corresponds to a clockwise rotation of the labels. For $f : \gamma_2 \to \mathbb{R}$ and $g : \gamma_3 \to \mathbb{R}$, we then define $\mathcal{E}_{m}^{[2]}(f)$ and $\mathcal{E}_{m}^{[3]}(g)$ as follows: 
	\begin{align}
		\label{eq:em1-em2-def}
		\mathcal{E}_m^{[2]}(f) := 2^{\frac{m}{2}} \mathcal{E}_{m}^{[1]}(f \circ R) \circ R^{-1} \quad \text{and} \quad \mathcal{E}_m^{[3]}(g) := \mathcal{E}_{m}^{[1]}(g \circ R^{-1}) \circ R.
	\end{align}	
	The properties of these operators are summarized in the following lemma.
	\begin{lemma}
		\label{lem:em3-derivative-interp-and-stability}
		Let $m \in \mathbb{N}_0$, $b \in C^{\infty}_c(\unitint)$ with $\int_{\unitint} b(t) \ dt = 1$, and $i \in \{1,2,3\}$. For all $(s, q) \in \mathcal{A}_m$ and $f \in W^{s-m-\frac{1}{q}, q}(\gamma_i)$, the lifting $\mathcal{E}_m^{[i]}(f) \in W^{s, q}(T)$, and there holds
		\begin{alignat}{2}
			\label{eq:em3-derivative-interp}
			\partial_{n}^j \mathcal{E}_m^{[i]}(f)|_{\gamma_i} &= f \delta_{jm}, \qquad & & j \in \{ 0, 1, \ldots, m \},
		\end{alignat}	
		and for real $0 \leq \beta \leq s$,
		\begin{align}
			\label{eq:em3-stability}
			\| \mathcal{E}_{m}^{[i]}(f) \|_{\beta, q, T} \lesssim_{b, m, \beta, q} \begin{cases}
				\| d_{i+1}^{m-\beta+\frac{1}{q}} f\|_{\gamma_i} & \text{if } 0 \leq \beta \leq m, \\
				\|f\|_{\beta-m-\frac{1}{q}, q, \gamma_i} & \text{if } {m+1 \leq \beta \leq s}, \ (\beta, q) \in \mathcal{A}_m,
			\end{cases}
		\end{align}
		where $d_{j}$ is the distance to $\bdd{a}_{j}$. If, in addition, $f \in \mathcal{P}_p(\gamma_i)$, $p \in \mathbb{N}_0$, then $\mathcal{E}_m^{[i]}(f) \in \mathcal{P}_{p+m}(T)$.	
	\end{lemma}
	\noindent \Cref{eq:em3-derivative-interp} shows that the function $\mathcal{E}_m^{[i]}(f)$ is a lifting of $f$ from $\gamma_i$ to $T$.
	The proof of \cref{lem:em3-derivative-interp-and-stability}, along with the rest of the results in this section, are postponed until \cref{sec:single-edge-stability}.
	
	\subsection[The first Mu\~{n}oz-Sola operator]{The Mu\~{n}oz-Sola operator $\mathcal{M}_{m, r}$}
	
	We now define a lifting operator motivated by Mu\~{n}oz-Sola \cite[Lemma 6]{Munoz97}. Given $m, r \in \mathbb{N}_0$, $b \in C_c^{\infty}(\unitint)$, and function $f : \unitint \to \mathbb{R}$, we define $\mathcal{M}_{m, r}^{[1]}(f)$ formally by the rule
	\begin{align*}
		\mathcal{M}_{m, r}^{[1]}(f)(x, y) := x^r \mathcal{E}_m^{[1]}(\tau^{-r} f)(x, y) = x^r \frac{(-y)^m}{m!} \int_{\unitint} b(t) \frac{f(x + tx)}{(x + ty)^r} \ dt,  \quad (x, y) \in T.
	\end{align*}
	Here, and in what follows, $\tau$ denotes the function $\tau(t) = t$ for $t \in \unitint$. We again use the notation $\mathcal{M}_{m, r}^{[1]}[b](f)$ and $\mathcal{M}_{m, r}^{[1]}(f) := \mathcal{M}_{m, r}^{[1]}(f \circ \varphi_1)$ for $f : \gamma_1 \to \mathbb{R}$ analogously as above. Loosely speaking, the presence of the term $(x + ty)^{-r}$ in the above expression means that, for $r > 0$, $f(t)$ needs to decay to 0 sufficiently fast at $t = 0$ for $\mathcal{M}_{m, r}^{[1]}(f)$ to have sufficient regularity. To characterize this decay more precisely, we introduce some additional spaces.
	
	Let $i \in \{1,2,3\}$ and let $W_L^{k+\beta, q}(\gamma_i)$, $k \in \mathbb{N}_0$, $0 \leq \beta < 1$, $1 < q < \infty$, denote the subspace of $W^{k+\beta, q}(\gamma_i)$ functions satisfying
	\begin{align}
		\label{eq:left-space-conditions}
		\begin{dcases}
			\partial_t^{j} f(\bdd{a}_{i+1}) = 0 & \text{for } 0 \leq j < k + \beta - \frac{1}{q}, \\
			\| d_{i+1}^{-\frac{1}{q}} \partial_t^k f\|_{q, \gamma_i} < \infty
			& \text{if } \beta q = 1,
		\end{dcases}
	\end{align}
	equipped with the norm
	\begin{align}
		\label{eq:left-norm-def-gamma}
		\leftnormsup{u}{k+\beta, q, \gamma_i}{q} := \|u\|_{k+\beta, q, \gamma_i}^q + \begin{dcases}
			\| d_{i+1}^{-\frac{1}{q}} \partial_t^k f\|_{q, \gamma_i}^q
			& \text{if } \beta q =1, \\
			0 & \text{otherwise},
		\end{dcases}
	\end{align}		
	where $d_j$ is defined in \cref{lem:em3-derivative-interp-and-stability}.
	The weighted spaces $W_L^{k+\beta, q}(\gamma_i)$ are crucial for characterizing the continuity of the operators $\mathcal{M}_{m, r}^{[i]}(f)$, $i \in \{1,2,3\}$, where $\mathcal{M}_{m, r}^{[2]}(f)$ and $\mathcal{M}_{m, r}^{[3]}(f)$ are defined analogously as in \cref{eq:em1-em2-def}, as the following result shows.
	\begin{lemma}
		\label{lem:mmr3-derivative-interp}
		Let $m \in \mathbb{N}_0$, $r \in \mathbb{N}$, $b \in C^{\infty}_c(\unitint)$ with $\int_{\unitint} b(t) \ dt = 1$, and $i \in \{1,2,3\}$. For all $(s, q) \in \mathcal{A}_m$ and $f \in W^{s-m-\frac{1}{q}, q}(\gamma_i) \cap W^{\min\{s-m, r\}-\frac{1}{q}, q}_L(\gamma_i)$, the lifting $\mathcal{M}_{m, r}^{[i]}(f) \in W^{s, q}(T)$, and there holds
		\begin{subequations}
			\begin{alignat}{2}
				\label{eq:mmr3-derivative-interp}
				\partial_{n}^j \mathcal{M}_{m, r}^{[i]}(f)|_{\gamma_i} &= f \delta_{jm}, \qquad & &j \in \{0, 1, \ldots, m\}, \\
				\partial_{n}^l \mathcal{M}_{m, r}^{[i]}(f)|_{\gamma_{i+2}} &= 0, \qquad & &l \in \{0, 1, \ldots, r-1\} \text{ and } (s-l)q > 1,
			\end{alignat}
		\end{subequations}		
		and for real $0 \leq \beta \leq s$, 
		\begin{align}
			\label{eq:mmr3-stability}
			\|\mathcal{M}_{m, r}^{[i]}(f)\|_{\beta, q, T} \lesssim_{b, m, r, \beta, q} \begin{cases}
				\|d_{i+1}^{m-\beta+\frac{1}{q}} f \|_{q, \gamma_i} & \text{if } 0 \leq \beta \leq m, \\
				\leftnorm{f}{\beta-m-\frac{1}{q}, q, \gamma_i} & \text{if } {m+1 \leq \beta \leq m+r}, \ (\beta, q) \in \mathcal{A}_m, \\
				\|f\|_{\beta-m-\frac{1}{q}, q, \gamma_i} & \text{if } m + r < \beta \leq s, \ (\beta, q) \in \mathcal{A}_m.
			\end{cases}
		\end{align}		
		If, additionally, $f \in \mathcal{P}_p(\gamma_i)$, $p \in \mathbb{N}_0$, with $\partial_t^l f(\bdd{a}_1) = 0$ for $l \in \{0, 1, \ldots, r-1\}$, then $\mathcal{M}_{m, r}^{[i]}(f) \in \mathcal{P}_{p+m}(T)$.
	\end{lemma}
	\noindent In particular, the function $\mathcal{M}_{m, r}^{[i]}(f)$ is a lifting of $f$ with the additional property that the normal derivatives up to order $r-1$ of $\mathcal{M}_{m, r}^{[i]}(f)$ vanish on $\gamma_{i+2}$.
	
	\subsection[The second Mu\~{n}oz-Sola operators]{The Mu\~{n}oz-Sola operator $\mathcal{S}_{m, r}$}
	
	We define one final lifting operator, again inspired by  Mu\~{n}oz-Sola \cite[Lemmas 7 \& 8]{Munoz97}: Let $m$, $r$, $b$, and $f$ be as above and define $\mathcal{S}_{m, r}^{[1]}(f)$ formally by the rule
	\begin{align*}
		\mathcal{S}_{m, r}^{[1]}(f)(x, y) &:= \{ x(1-x-y) \}^r \mathcal{E}_m^{[1]}\left( \frac{f}{\{\tau (1- \tau) \}^r} \right)(x, y) \\
		&= \{ x(1-x-y) \}^r \frac{(-y)^m}{m!} \int_{\unitint} b(t) \left. \frac{f(s)}{\{s(1-s) \}^r} \right|_{s = x + ty} \ dt,  \qquad (x, y) \in T,
	\end{align*}
	and again use the notation $\mathcal{S}_{m, r}^{[1]}[b](f)$ and $\mathcal{S}_{m, r}^{[1]}(f) := \mathcal{S}_{m, r}^{[1]}(f \circ \varphi_3)$ for $f : \gamma_1 \to \mathbb{R}$ analogously as above. Similarly to the operator $\mathcal{M}_{m, r}^{[1]}$, the presence of the term $\{(x + ty)(1-x-ty)\}^{-r}$ in the above expression means that for $r > 0$, $f(t)$ needs to decay to 0 sufficiently fast at $t = 0$ and $t=1$ for $\mathcal{S}_{m, r}^{[1]}(f)$ to have sufficiently regularity. 
	
	Here, the appropriate space to describe this decay is $W_{00}^{k+\beta, q}(\gamma_i)$, $k \in \mathbb{N}_0$, $0 \leq \beta < 1$, $1 < q < \infty$, $i \in \{1,2,3\}$, the subspace of $W^{k+\beta, q}(\gamma_i)$ functions satisfying
	\begin{align}
		\label{eq:zz-space-conditions}
		\begin{dcases}
			\partial_t^{i} f|_{\partial \gamma_i} = 0 & \text{for } 0 \leq i < k + \beta - \frac{1}{q}, \\
			\| (d_{i+1} d_{i+2})^{-\frac{1}{q}} \partial_t^k f\|_{q, \gamma_i} < \infty & \text{if } \beta q = 1,
		\end{dcases}
	\end{align}
	equipped with the norm
	\begin{align*}
		\zznormsup{u}{k+\beta, q, \gamma_i}{q} := \|u\|_{k+\beta, q, \gamma_i}^q + \begin{dcases}
			\| (d_{i+1} d_{i+2})^{-\frac{1}{q}} \partial_t^k f\|_{q, \gamma_i}^q & \text{if } \beta q = 1, \\
			0 & \text{otherwise},
		\end{dcases}
	\end{align*}
	We then have the following result for the operator $\mathcal{S}_{m, r}^{[i]}$, $i \in \{1,2,3\}$, where $\mathcal{S}_{m, r}^{[2]}(f)$ and $\mathcal{S}_{m, r}^{[3]}(f)$ are defined analogously as in \cref{eq:em1-em2-def}.
	\begin{lemma}
		\label{lem:smr3-derivative-interp}
		Let $m \in \mathbb{N}_0$, $r \in \mathbb{N}$, $b \in C^{\infty}_c(\unitint)$ with $\int_{\unitint} b(t) \ dt = 1$, and $i \in \{1,2,3\}$. For all $(s, q) \in \mathcal{A}_m$ and $f \in W^{s-m-\frac{1}{q}, q}(\gamma_i) \cap W^{\min\{s-m, r\}-\frac{1}{q}, q}_{00}(\gamma_i)$, the lifting $\mathcal{S}_{m, r}^{[i]}(f) \in W^{s, q}(T)$, and there holds
		\begin{subequations}
			\label{eq:smr3-derivative-interp}
			\begin{alignat}{2}
				\partial_{n}^j \mathcal{S}_{m, r}^{[i]}(f)|_{\gamma_i} &= f \delta_{jm}, \qquad & & j \in \{0, 1, \ldots, m\}, \\
				\partial_{n}^l \mathcal{S}_{m, r}^{[i]}(f)|_{\gamma_{i+1} \cup \gamma_{i+2}} &= 0, \qquad & & l \in \{0, 1, \ldots, r-1\} \text{ and } (s-l)q>1,
			\end{alignat}
		\end{subequations}		
		and for real $0 \leq \beta \leq m$, 
		\begin{align}
			\label{eq:smr3-stability-low}
			\|\mathcal{S}_{m, r}^{[i]}(f)\|_{\beta, q, T} \lesssim_{b, m, r, \beta, q} 	\|d_{i+1}^{m-\beta+\frac{1}{q}} f\|_{q, \gamma_i} + \|d_{i+2}^{m-\beta+\frac{1}{q}} f\|_{q, \gamma_i}, 
		\end{align}
		while for ${m + 1 \leq \beta \leq s}$,
		\begin{align}
			\label{eq:smr3-stability-high}
			\|\mathcal{S}_{m, r}^{[i]}(f)\|_{\beta, q, T} \lesssim_{b, m, r, \beta, q} \begin{cases}
				\zznorm{f}{\beta-m-\frac{1}{q}, q, \gamma_i} & \text{if } \beta \leq m + r, \ (\beta, q) \in \mathcal{A}_m, \\
				\|f\|_{\beta - m - \frac{1}{q}, q, \gamma_i} & \text{if } m + r < \beta \leq s, \ (\beta, q) \in \mathcal{A}_m.
			\end{cases} 
		\end{align}
		If, additionally, $f \in \mathcal{P}_p(\gamma_i)$, $p \in \mathbb{N}_0$, with $\partial_t^i f|_{\partial \gamma_i} = 0$ for $i \in \{0, 1, \ldots, r-1\}$, then $\mathcal{S}_{m, r}^{[i]}(f) \in \mathcal{P}_{p+m}(T)$.
	\end{lemma}
	\noindent  In particular, the function $\mathcal{S}_{m, r}^{[i]}(f)$ is a lifting of $f$ with the additional property that the normal derivatives up to order $r-1$ of $\mathcal{S}_{m, r}^{[i]}(f)$ vanish on $\gamma_{i+1}$ and $\gamma_{i+2}$.
	
	\section[Construction of the Lifting Operator]{Construction of the Lifting Operator $\tilde{\mathcal{L}}$}
	\label{sec:constructing-lifting}
	
	In this section we explicitly construct the operator $\tilde{\mathcal{L}}$ in \cref{thm:tilde-l1-existence} using the single edge operators in the previous section. The construction proceeds in three steps, one per edge. Throughout this section, let $b \in C_c^{\infty}(\unitint)$ denote any fixed function satisfying $\int_\unitint b(t) \ dt = 1$. 
	
	\subsection[Stable Lifting from One Edge]{Stable Lifting from $\gamma_1$}
	
	We begin by constructing a lifting operator from $\gamma_1$. Given functions $f, g : \gamma_1 \to \mathbb{R}$, we formally define $\tilde{\mathcal{L}}^{[1]} : T \to \mathbb{R}$ by the rule
	\begin{align*}
		\tilde{\mathcal{L}}^{[1]}(f, g) := \mathcal{E}_0^{[1]}[b](f) + \mathcal{E}_1^{[1]}[b]\left(g - \partial_{n} \mathcal{E}_0^{[1]}[b](f)|_{\gamma_1}  \right) \qquad \text{on } T. 
	\end{align*}
	The following lemma shows that the operator $\tilde{\mathcal{L}}^{[1]}$ is a stable lifting of $f$ and $g$.
	\begin{lemma}
		\label{lem:single-edge-combined-extension}
		For all $(s, q) \in \mathcal{A}_1$, $f \in W^{s-\frac{1}{q}, q}(\gamma_1)$, and $g \in W^{s-1-\frac{1}{q}, q}(\gamma_1)$, there holds		
		\begin{align}
			\label{eq:l1-interp-data}
			\tilde{\mathcal{L}}^{[1]}(f, g)|_{\gamma_1} = f \quad \text{and} \quad \partial_{n} \tilde{\mathcal{L}}^{[1]}(f, g)|_{\gamma_1} = g
		\end{align}
		with
		\begin{align}
			\label{eq:l1-stability}
			\|\tilde{\mathcal{L}}^{[1]}(f, g)\|_{s, q, T} \lesssim_{s, q} \|f\|_{s-\frac{1}{q}, q, \gamma_1} + \|g\|_{s-1 - \frac{1}{q}, q, \gamma_1}.
		\end{align}
		Moreover, if $f \in \mathcal{P}_{p}(\gamma_1)$ and $g \in \mathcal{P}_{p-1}(\gamma_1)$, $p \in \mathbb{N}_0$, then $\tilde{\mathcal{L}}^{[1]}(f, g) \in \mathcal{P}_{p}(T)$.
	\end{lemma}
	\begin{proof}
		Let $(s, q) \in \mathcal{A}_1$, $f \in W^{s-\frac{1}{q}, q}(\gamma_1)$, and $g \in W^{s-1-\frac{1}{q}, q}(\gamma_1)$ be given. \Cref{eq:l1-interp-data} follows immediately from \cref{eq:em3-derivative-interp}. Moreover, \cref{eq:em3-stability} and the trace theorem \cref{eq:trace-thm} give
		\begin{align*}
			\|\tilde{\mathcal{L}}^{[1]}(f, g)\|_{s, q, T} &\lesssim_{s, q} \|f\|_{s-\frac{1}{q}, q, \gamma_1} + \|g - \partial_{n} \mathcal{E}_0^{[1]}[b](f)\|_{s-1-\frac{1}{q}, q, \gamma_1} \\
			&\lesssim_{k, q}  \|f\|_{s-\frac{1}{q}, q, \gamma_1} + \|g\|_{s-1 - \frac{1}{q}, q, \gamma_1} + \|\mathcal{E}_0^{[1]}[b](f) \|_{s, q, T} \\
			&\lesssim_{k, q} \|f\|_{s-\frac{1}{q}, q, \gamma_1} + \|g\|_{s-1 - \frac{1}{q}, q, \gamma_1}.
		\end{align*}
		
		Now let $f \in \mathcal{P}_{p}(\gamma_1)$ and $g \in \mathcal{P}_{p-1}(\gamma_1)$, $p \in \mathbb{N}_0$. Then, $\mathcal{E}_0^{[1]}[b](f) \in \mathcal{P}_p(T)$ by \cref{lem:em3-derivative-interp-and-stability}, and so $\partial_{n} \mathcal{E}_0^{[1]}[b](f)|_{\gamma_1} \in \mathcal{P}_{p-1}(\gamma_1)$. Appealing to \cref{lem:em3-derivative-interp-and-stability} again shows that $\mathcal{E}_1^{[1]}[b](g - \partial_{n} \mathcal{E}_0^{[1]}[b](f)|_{\gamma_1}) \in \mathcal{P}_{p}(T)$, and so $\tilde{\mathcal{L}}^{[1]}(f, g) \in \mathcal{P}_p(T)$.
	\end{proof}

	\subsection[Stable Lifting from Two Edges]{Stable Lifting from $\gamma_1$ and $\gamma_2$}
	
	With the aid of the operator $\tilde{\mathcal{L}}^{[1]}$, we proceed counterclockwise around $\partial T$ and construct a lifting operator from $\gamma_1 \cup \gamma_2$. For $f, g : \gamma_1 \cup \gamma_2 \to \mathbb{R}$, we formally define $\mathcal{K}^{[2]}(f, g), \tilde{\mathcal{L}}^{[2]}(f, g) : T \to \mathbb{R}$ by the rules
	\begin{subequations}
		\label{eq:two-edge-lifting-def}
		\begin{alignat}{2}
			\mathcal{K}^{[2]}(f, g)  &:= \tilde{\mathcal{L}}^{[1]}(f, g) + \mathcal{M}_{0, 2}^{[2]}[b](f_2 - \tilde{\mathcal{L}}^{[1]}(f, g)|_{\gamma_2}) \qquad & & \text{on } T, \\
			\tilde{\mathcal{L}}^{[2]}(f, g) &:= \mathcal{K}^{[2]}(f, g) + \mathcal{M}_{1, 2}^{[2]}[b](g_2 - \partial_{n} \mathcal{K}^{[2]}(f, g)|_{\gamma_2} ) \qquad & & \text{on } T. 
		\end{alignat}
	\end{subequations}
	The operator $\mathcal{K}^{[2]}$ corrects the trace of $\tilde{\mathcal{L}}^{[1]}(f, g)$ on $\gamma_2$ to be $f_2$ without changing the trace or normal derivative of $\tilde{\mathcal{L}}^{[1]}(f, g)$ on $\gamma_1$, while $\tilde{\mathcal{L}}^{[2]}(f, g)$ corrects the normal derivative of $\mathcal{K}^{[2]}(f, g)$ on $\gamma_2$ without changing the trace or normal derivative of $\mathcal{K}^{[2]}(f, g)$ on $\gamma_1$ or its trace on $\gamma_2$.

	The continuity of the operators $\mathcal{M}_{0, 2}^{[2]}$ and $\mathcal{M}_{1, 2}^{[2]}$ appearing in \cref{eq:two-edge-lifting-def} depends on the weighted spaces $W^{s, q}_L(\gamma_2)$ \cref{eq:mmr3-stability}. The following lemma provides a useful criterion for verifying when a pair of traces belongs to this space.
	\begin{lemma}
		\label{lem:trace-diff-left-space}
		Let $(s, q) \in \mathcal{A}_1$ and $(f^0, f^1) \in X^{s, q}(\partial T)$. Suppose that for some $j \in \{1,2,3\}$ and $n \in \{0, 1\}$, there holds
		\begin{enumerate}
			\item[(i)] $f_{j}^0 = f_{j}^1 = 0$, and
			
			\item[(ii)] $f_{j+1}^0 = 0$ if $n=1$.
		\end{enumerate}
		Then, $f_{j+1}^n \in W_L^{\beta - \frac{1}{q}, q}(\gamma_{j+1})$ with $\beta = \min\{ s-n, 2 \}$, and there holds
		\begin{align}
			\label{eq:trace-diff-left-space-bound}
			\leftnorm{f_{j+1}^{n}}{\beta-\frac{1}{q}, q, \gamma_{j+1}} &\lesssim_{s, q}   \| (f^0, f^1) \|_{X^{s, q}, \partial T}.
		\end{align}
		If, in addition, $f_{j+1}^l \in \mathcal{P}_{p-l}(\gamma_i)$, $l \in \{0, 1\}$, for some $p \in \mathbb{N}_0$ and \cref{eq:review:sigma0-cont-1,eq:review:sigma1-cont,eq:review:sigma1-deriv-cont} hold for $i=j+2$, then $\partial_t^l f^n_{j+1}(\bdd{a}_{j+2}) = 0$, $l \in \{0, 1\}$.
	\end{lemma}
	\begin{proof}
		Let $F = (f^0, f^1) \in X^{s, q}(\partial T)$ be as in the statement of the lemma. By definition, there holds
		\begin{align}
			\label{eq:proof:trace-comp-triangle-ineq}
			\|f_{j+1}^{n} \|_{\beta-\frac{1}{q}, q, \gamma_{j+1}} &\lesssim_{s, q}   \| (f^0, f^1) \|_{X^{s, q}, \partial T},
		\end{align}
		and so it remains to verify the conditions in \cref{eq:left-space-conditions} and bound the weighted $L^q$ norm term in \cref{eq:left-norm-def-gamma} when $s - 2/q \in \mathbb{Z}$.
		
		Suppose first that $n=0$. Thanks to (i), we have
		\begin{align*}
			f_{j+1}^0 \circ \phi_{j+1} &= f_{j+1}^0 \circ \phi_{j+1} - f_j^0 \circ \phi_{j},
		\end{align*}
		where $\phi_{j}(h) = \bdd{a}_{j+2} - h \unitvec{t}_{j}$ and $\phi_{j+1}(h) = \bdd{a}_{j+2} + h \unitvec{t}_{j+1}$ for $0 \leq h \leq 1$ are (partial) parametrizations of $\gamma_j$ and $\gamma_{j+1}$. Thus, $f_{j+1}^0(\bdd{a}_{j+2}) = 0$. Using (i) once again gives
		\begin{align*}
			\partial_{h} \{ f_{j+1}^0 \circ \phi_{j+1} \}   =  \unitvec{t}_{j+1} \cdot \{  \sigma_{j+1}^1(F) \circ \phi_{j+1} \} =  \unitvec{t}_{j+1} \cdot \{  \sigma_{j+1}^1(F) \circ \phi_{j+1} -  \ \sigma_{j}^1(F) \circ \phi_{j} \}. 
		\end{align*}
		Thus, $\partial_t f_{j+1}^{0} (\bdd{a}_{j+2}) = 0$ when $(s-1)q > 2$ by \cref{eq:review:sigma1-cont}. For $(s-1)q = 2$, we use a change of variable and the triangle inequality to conclude
		\begin{align*}
			\| d_{j+2}^{-\frac{1}{q}} \partial_t f^0\|_{q, \gamma_{j+1}}^q
			\lesssim_q  \|\partial_t f^0\|_{q, \gamma_{j+1}}^q + \mathcal{I}_{j+2}^q( \sigma_{j}(F), \sigma_{j+1}(F) ). 
		\end{align*}
		\Cref{eq:trace-diff-left-space-bound} now follows from \cref{eq:proof:trace-comp-triangle-ineq} and $f_{j+1}^{0} \in W_L^{\min\{s, 2\} - \frac{1}{q}, q}(\gamma_{j+1})$ by \cref{eq:left-space-conditions}.

		Now suppose that $n=1$. Since $\unitvec{t}_j$ and $\unitvec{t}_{j+1}$ are linearly independent, there exist constants $c_0, c_1 \in \mathbb{R}$ such that $\unitvec{n}_{j+1} = c_0 \unitvec{t}_j + c_1 \unitvec{t}_{j+1}$. Thanks to the orthogonality of $\unitvec{t}_{j+1}$ and $\unitvec{n}_{j+1}$, there holds
		\begin{multline*}
			f^1_{j+1} =  \unitvec{n}_{j+1} \cdot  \sigma_{j+1}^{1}(F) = c_0 \unitvec{t}_j \cdot \sigma_{j+1}^{1}(F)  + c_1 \unitvec{t}_{j+1} \cdot \sigma_{j+1}^{1}(F) \\
			= c_0 \unitvec{t}_j \cdot \sigma_{j+1}^{1}(F)  + c_1 \partial_t \sigma_{j+1}^{0}(f^0) = c_0 \unitvec{t}_j \cdot \sigma_{j+1}^{1}(F),
		\end{multline*}
		where we used condition (ii) in the final equality. Again using $F \equiv 0$ on $\gamma_j$ by (i), we obtain 
		\begin{align*}
			\partial_{h}^{i} \{ f_{j+1}^1  \circ \phi_{j+1} \}  =  \alpha \partial_{h}^{i}  \{ \unitvec{t}_j \cdot \sigma_{j+1}^{1}(F) \circ \phi_{j+1}   - \unitvec{t}_{j+1} \cdot \sigma_{j}^{1}(F) \circ \phi_{j} \}, \quad i \in \{0, 1\}.
		\end{align*}
		Consequently, $\partial_t^{i} f_{j+1}^{1}(\bdd{a}_{j+2}) = 0$ when $s > i + 1 + \frac{2}{q}$ by \cref{eq:review:sigma1-cont,eq:review:sigma1-deriv-cont}. Similar arguments as above show that for $(s-i-1)q=2$, there holds
		\begin{align*}
			\|d_{j+1}^{-\frac{1}{q}} \partial_t^i f \|_{1, \gamma_{j+1}}^q \lesssim \|\partial_t^i f^1\|_{q, \gamma_{j+1}}^q 
			+ \begin{cases}
				\mathcal{I}_{j+2}^q( \sigma_{j}(F), \sigma_{j+1}(F) ) & \text{if } i = 0, \\
				\mathcal{I}_{j+2}^q( \unitvec{t}_{j+1} \cdot \sigma_{j}(F), \unitvec{t}_{j} \cdot \sigma_{j+1}(F) ) & \text{if } i = 1.
			\end{cases} 
		\end{align*}
		Thus, $f_{j+1}^{1} \in W_L^{\min\{ s-1, 2 \} - \frac{1}{q}, q}(\gamma_{j+1})$ with \cref{eq:trace-diff-left-space-bound}.
		
		Now suppose that $f_i^l \in \mathcal{P}_{p}(\gamma_i)$, $l \in \{0, 1\}$, $i \in \{j, j+1\} $ for some $p \in \mathbb{N}_0$ and \cref{eq:review:sigma0-cont-1,eq:review:sigma1-cont,eq:review:sigma1-deriv-cont} hold for $i=j+2$. Then, we have already shown that $f^n_{j+1} \in W_L^{2 - \frac{1}{q},q}(\gamma_{j+1})$ for all $1 < q < \infty$, and so $\partial_t^l f^n_{j+1} (\bdd{a}_{j+2}) = 0$, $l \in \{0, 1\}$.
	\end{proof}

	\begin{lemma}
		\label{lem:two-edge-combined-extension}
		For all $(s, q) \in \mathcal{A}_1$ and  $(f, g) \in X^{s, q}(\partial T)$, there holds
		\begin{align}
			\label{eq:l2-interp-data}
			\tilde{\mathcal{L}}^{[2]}(f, g)|_{\gamma_i} = f_i \quad \text{and} \quad \partial_{n} \tilde{\mathcal{L}}^{[2]}(f, g)|_{\gamma_i} = g_i, \qquad i \in \{1,2\},
		\end{align}
		and
		\begin{align}
			\label{eq:l2-stability}
			\|\tilde{\mathcal{L}}^{[2]}(f, g)\|_{s, q, T} \lesssim_{s, q} \|(f, g)\|_{X^{s, q}, \partial T}.
		\end{align}
		If, in addition, $f_i \in \mathcal{P}_{p}(\gamma_i)$, $g_i \in \mathcal{P}_{p-1}(\gamma_i)$, $i \in \{1, 2\} $, $p \in \mathbb{N}_0$, and \cref{eq:review:sigma0-cont-1,eq:review:sigma1-cont,eq:review:sigma1-deriv-cont} hold for $i=3$, then $\tilde{\mathcal{L}}^{[2]}(f, g) \in \mathcal{P}_{p}(T)$.
	\end{lemma}
	\begin{proof}
		Let  $(s, q) \in \mathcal{A}_1$ and $(f, g) \in X^{s, q}(\partial T)$. Applying \cref{lem:single-edge-combined-extension,lem:trace-diff-left-space} with $n=0$ gives $f_{2} - \tilde{\mathcal{L}}^{[1]}(f, g)|_{\gamma_2} \in W_L^{\min\{s, 2\}-\frac{1}{q}, q}(\gamma_2)$, and so 
		\begin{align*}
			\mathcal{K}^{[2]}(f, g)|_{\gamma_i} = f_{i}, \quad i \in \{1, 2\}, \quad \text{and} \quad \partial_{n} \mathcal{K}^{[2]}(f, g)|_{\gamma_1} = g_1
		\end{align*}
		with the estimate
		\begin{align*}
			\| \mathcal{K}^{[2]}(f, g) \|_{s, q, T} &\lesssim \| \tilde{\mathcal{L}}^{[1]}(f, g)\|_{s, q, T} + \| \mathcal{M}_{0, 2}^{[2]}(f_1 - \tilde{\mathcal{L}}^{[1]}(f, g)|_{\gamma_1}) \|_{s, q, T} \\
			&\lesssim \|(f, g)\|_{X^{s, q}, \partial T}
		\end{align*}
		by \cref{lem:mmr3-derivative-interp}.
		Now applying \cref{lem:single-edge-combined-extension,lem:trace-diff-left-space} with $n=1$ shows that $g_{2} - \partial_n \mathcal{K}_1(f, g)|_{\gamma_2} \in W_L^{\min\{s-1, 2\}-\frac{1}{q}, q}(\gamma_2)$. Another application of \cref{lem:mmr3-derivative-interp} and the triangle inequality completes the proof of \cref{eq:l2-interp-data,eq:l2-stability}.
		
		Now assume further that $f_i \in \mathcal{P}_p(\gamma_i)$, $g_i \in \mathcal{P}_{p-1}(\gamma_i)$, $i\in \{1,2\}$, $p \in \mathbb{N}_0$, and \cref{eq:review:sigma0-cont-1,eq:review:sigma1-cont,eq:review:sigma1-deriv-cont} hold for $i=3$. Thanks to \cref{lem:single-edge-combined-extension}, $\tilde{\mathcal{L}}^{[1]}(f, g) \in \mathcal{P}_{p}(T)$, and \cref{lem:trace-diff-left-space} then gives $\partial_t^{l} (f_{2} - \tilde{\mathcal{L}}^{[1]}(f, g)|_{\gamma_2})(\bdd{a}_3) = 0$ for $l \in \{0, 1\}$. Consequently, $\mathcal{K}^{[2]}(f, g) \in \mathcal{P}_{p}(T)$ by \cref{lem:mmr3-derivative-interp}. Applying similar arguments show that  $\partial_t^{l} (g_2 - \partial_n \mathcal{K}^{[2]}(f, g)|_{\gamma_2})(\bdd{a}_3) = 0$ for $l \in \{0, 1\}$, and so \cref{lem:mmr3-derivative-interp} gives $\tilde{\mathcal{L}}^{[2]}(f, g) \in \mathcal{P}_{p}(T)$.
	\end{proof}
	
	\subsection{Stable Lifting from Entire Boundary}
	
	With the aid of the operator $\tilde{\mathcal{L}}^{[2]}$, we again proceed counterclockwise around $\partial T$ and finally complete the construction of the lifting operator $\tilde{\mathcal{L}}$. For $f, g : \partial T \to \mathbb{R}$, we formally define $\mathcal{K}^{[3]}(f, g), \tilde{\mathcal{L}}(f, g) : T \to \mathbb{R}$ by the rules
	\begin{alignat*}{2}
		\mathcal{K}^{[3]}(f, g)  &:= \tilde{\mathcal{L}}^{[2]}(f, g) + \mathcal{S}_{0, 2}^{[3]}[b](f_3 - \tilde{\mathcal{L}}^{[2]}(f, g)|_{\gamma_3}) \qquad & & \text{on } T, \\
		\tilde{\mathcal{L}}(f, g) &:= \mathcal{K}^{[3]}(f, g) + \mathcal{S}_{1, 2}^{[3]}[b](g_3 - \partial_{n} \mathcal{K}^{[3]}(f, g)|_{\gamma_3} ) \qquad & & \text{on } T. 
	\end{alignat*}
	The operator $\mathcal{K}^{[3]}$ corrects the trace of $\tilde{\mathcal{L}}^{[2]}(f, g)$ on $\gamma_3$ to be $f_3$ without changing the trace or normal derivative of $\tilde{\mathcal{L}}^{[2]}(f, g)$ on $\gamma_1 \cup \gamma_2$, while $\tilde{\mathcal{L}}(f, g)$ corrects the normal derivative of $\mathcal{K}^{[3]}(f, g)$ on $\gamma_3$ without changing the trace or normal derivative of $\mathcal{K}^{[3]}(f, g)$ on $\gamma_1 \cup \gamma_2$ or its trace on $\gamma_3$. We start with an analogue of \cref{lem:trace-diff-left-space}.
	\begin{lemma}
		\label{lem:trace-diff-zz-space}
		Let $(s, q) \in \mathcal{A}_1$, and $(f^0, f^1) \in X^{s, q}(\partial T)$. Suppose that for some $n \in \{0, 1\}$, there holds
		\begin{enumerate}
			\item[(i)] $f_{i}^{l} =0$ for $l \in \{0, 1\}$ and $i \in \{1, 2\}$, and
			
			\item[(ii)] $f_{3}^{0} = 0$ if $n=1$.
		\end{enumerate}
		Then, $f_{3}^{n} \in W_{00}^{\beta - \frac{1}{q}, q}(\gamma_{3})$ with $\beta = \min\{ s-n, 2 \}$, and there holds
		\begin{align}
			\label{eq:trace-diff-zz-space-bound}
			\zznorm{f_{3}^{n}}{\beta-\frac{1}{q}, q, \gamma_{3}} &\lesssim_{s, q}   \| F \|_{X^{s, q},\partial T}.
		\end{align}
		If, in addition, $f_i^l \in \mathcal{P}_{p}(\gamma_i)$, $l \in \{0, 1\}$, $i \in \{0, 1, 2\} $ for some $p \in \mathbb{N}_0$ and \cref{eq:review:sigma0-cont-1,eq:review:sigma1-cont,eq:review:sigma1-deriv-cont} hold, then $\partial_t^l f^n_{3}|_{\partial \gamma_{3}} = 0$, $l \in \{0, 1\}$.
	\end{lemma}
	\begin{proof}
		Let $F = (f^0, f^1) \in X^{s, q}(\partial T)$ be as in the statement of the lemma and $n \in \{0, 1\}$. Applying \cref{lem:trace-diff-left-space} to $\gamma_2 \cup \gamma_3$ gives $f_{3}^{n} \in W_{L}^{\beta - \frac{1}{q}, q}(\gamma_{3})$ with 
		\begin{align}
			\label{eq:proof:gamma3-left-norm-bound}
			\leftnorm{f_{3}^{n}}{\beta-\frac{1}{q}, q, \gamma_{3}} &\lesssim_{s, q}   \| F \|_{X^{s, q},\partial T}.
		\end{align}
		Applying the same arguments as in the proof of \cref{lem:trace-diff-left-space} with $j=3$ and reversing the roles of $\gamma_3$ and $\gamma_{1}$ then gives
		\begin{align*}
			\partial_t^i f_{3}^{n} (\bdd{a}_{2}) = 0 \qquad \text{for } 0 \leq i <  \min\left\{ s - n - \frac{2}{q}, 2 \right\}, 
		\end{align*}
		and if  $(s-i-n)q = 2$ for some $i \in \{0, 1\}$, then
		\begin{align}
			\label{eq:proof:gamma3-right-norm-bound}
			 \| d_2^{-\frac{1}{q}} \partial_t^i f\|_{q, \gamma_3} 
			 \lesssim_{s, q} \|F\|_{X^{s, q}, \partial T}.
		\end{align}
		The inclusion $f_{3}^{n} \in W_{00}^{\beta - \frac{1}{q}, q}(\gamma_{3})$ then follows from \cref{eq:zz-space-conditions} on noting that $d_{1} + d_2 \lesssim d_1 d_2$ on $\gamma_3$, which in conjunction with \cref{eq:proof:gamma3-left-norm-bound}, \cref{eq:proof:gamma3-right-norm-bound}, and the triangle inequality gives \cref{eq:trace-diff-zz-space-bound}.
		
		Now suppose that $f_i^l \in \mathcal{P}_{p}(\gamma_i)$, $l \in \{0, 1\}$, $i \in \{1, 2, 3\} $ for some $p \in \mathbb{N}_0$ and \cref{eq:review:sigma0-cont-1,eq:review:sigma1-cont,eq:review:sigma1-deriv-cont} hold. Then, we have already shown that $f^n_{3} \in W_{00}^{2 - \frac{1}{q},q}(\gamma_3)$ for all $1 < q < \infty$, and so $\partial_t^l f^n_{3} (\bdd{a})|_{\gamma_3} = 0$, $l \in \{0, 1\}$.
	\end{proof}

	We now prove the main result, \cref{thm:tilde-l1-existence}.
	\begin{proof}[Proof of \cref{thm:tilde-l1-existence}]
		Let $(s, q) \in \mathcal{A}_1$ and $(f, g) \in X^{s, q}(\partial T)$. According to \cref{lem:two-edge-combined-extension} and \cref{lem:trace-diff-zz-space},  $f_{3} - \tilde{\mathcal{L}}^{[2]}(f, g)|_{\gamma_3} \in W_{00}^{\min\{s, 2\}-\frac{1}{q}, q}(\gamma_3)$. Consequently,  \cref{lem:smr3-derivative-interp} shows that
		\begin{align*}
			\mathcal{K}^{[3]}(f, g)|_{\gamma_i} = f_{i}, \quad i \in \{1, 2, 3\}, \quad \text{and} \quad \partial_{n} \mathcal{K}^{[3]}(f, g)|_{\gamma_j} = g_j, \quad j \in \{1, 2\},
		\end{align*}
		with the estimate
		\begin{align*}
			\| \mathcal{K}^{[3]}(f, g) \|_{s, q, T} &\lesssim_{s, q} \| \tilde{\mathcal{L}}^{[2]}(f, g)\|_{s, q, T} + \| \mathcal{S}_{0, 2}^{[3]}(f_3 - \tilde{\mathcal{L}}^{[2]}(f, g)|_{\gamma_3}) \|_{s, q, T} \\
			&\lesssim_{s, q} \|(f, g)\|_{X^{s, q}, \partial T}.
		\end{align*}
		Applying \cref{lem:two-edge-combined-extension} and \cref{lem:trace-diff-zz-space} once again show that $g_{3} - \partial_n \mathcal{K}^{[3]}(f, g)|_{\gamma_3} \in W_{00}^{\min\{s-1, 2\}-\frac{1}{q}, q}(\gamma_3)$. Another application of \cref{lem:smr3-derivative-interp} and the triangle inequality completes the proof of \cref{eq:tilde-l1-prop}.

		Now assume further that $f_i \in \mathcal{P}_p(\gamma_i)$, $g_i \in \mathcal{P}_{p-1}(\gamma_i)$, $i \in \{1,2, 3\}$, $p \in \mathbb{N}_0$, and \cref{eq:review:sigma0-cont-1,eq:review:sigma1-cont,eq:review:sigma1-deriv-cont} hold. Thanks to \cref{lem:single-edge-combined-extension}, $\tilde{\mathcal{L}}^{[2]}(f, g) \in \mathcal{P}_{p}(T)$, and \cref{lem:trace-diff-zz-space} then gives $\partial_t^{l} (f_{3} - \tilde{\mathcal{L}}^{[1]}(f, g)|_{\gamma_3} )|_{\partial \gamma_3} = 0$ for $l \in \{0, 1\}$. Consequently, $\mathcal{K}^{[3]}(f, g) \in \mathcal{P}_{p}(T)$ by \cref{lem:smr3-derivative-interp}. Applying similar arguments show that  $\partial_t^{l} (g_3 - \partial_n \mathcal{K}^{[3]}(f, g)|_{\gamma_3})|_{\partial \gamma_3}  = 0$ for $l \in \{0, 1\}$, and so \cref{lem:smr3-derivative-interp} gives $\tilde{\mathcal{L}}(f, g) \in \mathcal{P}_{p}(T)$.
	\end{proof}	
	
	\section{Continuity of single edge operators}
	\label{sec:single-edge-stability}
	
	In this section, we prove \cref{lem:em3-derivative-interp-and-stability,lem:mmr3-derivative-interp,lem:smr3-derivative-interp}.
	
	\subsection{Continuity in some weighted $L^q$ spaces}
	
	Given an open interval $\Lambda \subseteq (0, \infty)$, we define the weighted space $L^q(\Lambda; t^{\beta}dt)$, $\beta > -1$ to be the set of all measurable functions such that the following norm is finite:
	\begin{align}
		\label{eq:weighted-lq-norm-def}
		\| f \|_{q, \Lambda, \beta}^q := \int_{\Lambda} |f(t)|^q t^{\beta} \ dt.
	\end{align}
	The following result shows that $\mathcal{E}_m^{[1]}$ is well-defined on $L^q(\unitint; t^{mq+1} dt)$.
	\begin{lemma}
		\label{lem:weighted-lq-e0}
		For all $m \in \mathbb{N}_0$, $b \in C^{\infty}_c(\unitint)$, and $1 < q < \infty$, there holds
		\begin{alignat}{2}
			\label{eq:e0 l2-by-weighted-l2-bound}
			\| \mathcal{E}_m^{[1]}[b](f) \|_{q, T} &\leq  \frac{q}{(q-1)m!}  \|\tau^{-m} b\|_{\infty, \unitint} \| f\|_{q, \unitint, mq+1} \qquad \forall f \in L^q(\unitint; t^{mq+1} dt). 
		\end{alignat}
	\end{lemma}
	\begin{proof}
		Let $f \in L^q(\unitint; t^{mq+1} dt)$, $1 < q < \infty$, and $0 \leq x \leq 1$. Using that $y(x + ty)^{-1} < t^{-1}$ for $0 \leq y \leq 1-x$ and $0 < t < 1$, we obtain 
		\begin{align*}
			\left| y^m \int_{0}^{1} b(t) f(x+ty) \ dt \right|  &\leq \int_{0}^{1} |t^{-m} b(t)| (x + ty)^m |f(x+ty)| \ dt,
		\end{align*}
		and so
		\begin{multline*}
			\| \tau^{-m} b\|_{\infty, \unitint}^{-q} \int_{0}^{1-x} \left| y^m \int_{0}^{1} b(t) f(x+ty) \ dt \right|^q \ dy \\
			\leftstackrel{\substack{u=x+ty \\ z = x + y}}{\leq} \int_{x}^{1} \left( \frac{1}{z-x} \int_{x}^{z}  |u^m f(u)| \ du \right)^q \ dz \leq  \left( \frac{q}{q-1} \right)^q  \|f\|_{q, (x, 1), mq+1}^q
		\end{multline*}
		by Hardy's inequality \cite[Theorem 327]{Hardy52}. Additionally,
		\begin{align*}
		\int_{0}^{1}  \int_{x}^{1} |t^m f(t)|^q \ dt \ dx = \int_{0}^{1} |f(t)|^q t^{mq} \int_{0}^{t} \ dx \ dt =   \int_{0}^{1} | f(t)|^q  t^{mq+1} \ dt.
		\end{align*}
		\Cref{eq:e0 l2-by-weighted-l2-bound} now follows on collecting results. 
	\end{proof}
	\begin{remark}
		\label{rem:weighted-lq-bound}
		The same arguments show that \cref{eq:e0 l2-by-weighted-l2-bound} holds with $b$ replaced by $|b|$.
	\end{remark}
	
	\begin{lemma}
		\label{lem:emq1-hm-stability}
		For $m, r \in \mathbb{N}_0$, real $0 \leq \beta \leq m$, and $b \in C^{\infty}_c(\unitint)$, there holds
		\begin{align}
			\label{eq:emq1-hm-stability}
			\|\mathcal{M}_{m, r}^{[1]}[b](f)\|_{\beta, q, T} \lesssim_{b, m, r, \beta, q}
			\| f \|_{q, \unitint, (m-\beta)q+1} \qquad \forall f \in L^q(\unitint; t^{(m-\beta)q + 1}dt).
		\end{align}
	\end{lemma}
	\begin{proof}
		Let $m$, $r$ and $b$ be as in the statement of the lemma and $f \in C^{\infty}_c(\unitint)$. Let $j \in \{0, 1, \ldots, m\}$, $\alpha \in \mathbb{N}_0^2$ with $|\alpha| = j$, and $l_1 = \max(\alpha_1 - r, 0)$.  For $0 \leq i_1 \leq \alpha_1$ and $0 \leq i_2 \leq \alpha_2$, we apply the identities
		\begin{alignat*}{2}
			\partial_y \{ g(x + ty) \} &= t g'(x + ty) & &= t y^{-1} \partial_t \{ g(x + ty) \}, \\ 
			\partial_x \{ g(x + ty) \} &= g'(x + ty) & &= y^{-1} \partial_t \{ g(x + ty) \},
		\end{alignat*}
		and integrate by parts $i_1 + i_2 \leq k \leq m$ times to obtain
		\begin{align*}
			\int_{\unitint} b(t) \partial_{y}^{i_2} \partial_x^{i_1} \left\{ \frac{f(x + ty)}{(x+ty)^q} \right\} \ dt &= y^{-(i_1 + i_2)} \int_{\unitint} b(t) t^{i_2}  \partial_t^{i_1 + i_2}  \left\{ \frac{f(s)}{s^r} \middle\} \right|_{s = x + ty} \ dt, \\
			&= (-1)^{i_1 + i_2} y^{-(i_1 + i_2)} \int_{\unitint} \underbrace{\partial_t^{i_1 + i_2}\{ b(t) t^{i_2}\}}_{=b_{i_1, i_2}}  \frac{ f(x + ty)}{(x+ty)^r}  \ dt, \\
			&= (-1)^{i_1 + i_2} y^{-(i_1 + i_2)} \mathcal{E}_0^{[1]}[b_{i_1, i_2}]\left( \tau^{-r} f \right)(x, y).
		\end{align*}
		and so 		
		\begin{align*}
			& (-1)^{m} D^{\alpha} \mathcal{M}_{m, r}^{[1]}(f)(x, y) \\
			&\qquad = \sum_{\substack{0 \leq i_1 \leq \alpha_1 \\ 0 \leq i_2 \leq \alpha_2}} \binom{\alpha_1}{i_1} \binom{\alpha_2}{i_2} \partial_x^{\alpha_1 - i_1} \{ x^r \} \partial_y^{\alpha_2 - i_2} \left\{ \frac{y^m}{m!} \right\} \int_{\unitint} b(t) \partial_{y}^{i_2} \partial_x^{i_1} \left\{ \frac{f(x + ty)}{(x+ty)^q} \right\} \ dt, \\
			&\qquad=  \sum_{\substack{l_1 \leq i_1 \leq \alpha_1 \\ 0 \leq i_2 \leq \alpha_2}}  c_{r, \alpha, i_1, i_2} x^{r-\alpha_1 + i_1} y^{m- i_1 - \alpha_2} \mathcal{E}_0^{[1]}[b_{i_1, i_2}]\left( \tau^{-r} f \right)(x, y),
		\end{align*}
		where
		\begin{align*}
			b_{i_1, i_2} &:= \partial_t^{i_1 + i_2} \{ b(t) t^{i_2} \} \quad \text{and} \quad
			c_{r, \alpha, i_1, i_2} := \frac{ (-1)^{m + i_1 + i_2} \binom{\alpha_1}{i_1} \binom{\alpha_2}{i_2} r!}{(r-\alpha_1 + i_1)! (m-\alpha_2 + i_2)!}.
		\end{align*}
		Since $x \leq x + sy$ and $\frac{y}{x + sy} \leq s^{-1}$ for $(x, y) \in T$, $0 \leq s \leq 1$, there holds
		\begin{align*}
			\left| x^{r-\alpha_1 + i_1} y^{m-i_1-\alpha_2} \mathcal{E}_0^{[1]}[  b_{i_1, i_2} ] ( \tau^{-r} f ) \right| & \leq y^{m-i_1-\alpha_2} \int_{0}^{1} \left| b_{i_1, i_2}(t) \frac{f(x + ty)}{(x+ty)^{\alpha_1 - i_1}} \right| \ dt \\
			& \leq y^{m-j} \int_{0}^{1} | t^{i_1 - \alpha_1} b_{i_1, i_2}(t) f(x + ty) | \ dt.
		\end{align*}
		Since $b \in C^{\infty}_c(\unitint)$, the function $t^{i_1 - \alpha_1} b_{i_1, i_2} \in C^{\infty}_c(\unitint)$, and so \cref{eq:e0 l2-by-weighted-l2-bound,rem:weighted-lq-bound} gives
		\begin{align*}
			\| D^{\alpha} \mathcal{M}_{m, r}^{[1]}(f) \|_{q, T} \lesssim_{b, m, r, j, q}  \| \mathcal{E}_{m-j}^{[1]}[\tau^{i_1-\alpha_1} |b_{i_1, i_2}| ](|f|) \|_{q, T} \lesssim_{b, m, r, q} \|f\|_{q, \unitint,  (m-j)q+1}.
		\end{align*}	
		By density, $\mathcal{M}_{m, r}^{[1]}$ is a bounded operator from $L^q(\unitint; t^{(m-j)q + 1}dt)$ into $W^{j, q}(T)$. By the real method of interpolation (see e.g. \cite{BerLof76}), 
		\begin{align*}
			\mathcal{M}_{m, r}^{[1]} : [L^q(\unitint; t^{(m-j)q + 1}dt), L^q(\unitint; t^{(m-j-1)q + 1}dt)]_{\theta, q} \to  [W^{j, q}(T), W^{j+1, q}(T)]_{\theta, q}
		\end{align*} 
		is linear and continuous for any $0 \leq \theta \leq 1$ and $j \in \{0, 1, \ldots, m-1\}$. It is well-known that (see e.g. \cite[Theorem 5.4.1]{BerLof76})  
		\begin{align*}
			L^q(\unitint; t^{(m-j-\theta)q+1} \ dt) = [L^q(\unitint; t^{(m-j)q + 1}dt), L^q(\unitint; t^{(m-j-1)q + 1}dt)]_{\theta, q}
		\end{align*}
		and that (see e.g. \cite[Theorem 14.2.3]{Brenner08})
		\begin{align}
			\label{eq:proof:sobolev-interp-tri}
			W^{j+\theta, q}(T) = [W^{j, q}(T), W^{j+1, q}(T)]_{\theta, q}.
		\end{align}
		\Cref{eq:emq1-hm-stability} now follows.
	\end{proof}
	
	\subsection[The E operator]{The operator $\mathcal{E}_{m}$}
	\label{sec:em1-stability}
	
	The next result concerns the $W^{s, q}(T)$ stability of the operator $\mathcal{E}_m$.
	\begin{lemma}
		\label{lem:em-wsp-bound}
		Let $m \in \mathbb{N}_0$ and $b \in C_c^{\infty}(\unitint)$. For all $(s, q) \in \mathcal{A}_m$, there holds
		\begin{align}
			\label{eq:em-wsp-bound}
			\left\| \mathcal{E}_m[b]\left( f \right) \right\|_{s, q, T} \lesssim_{b, m, s, q} \|f\|_{s - m - \frac{1}{q}, q, \unitint} \qquad \forall f \in W^{s-m-\frac{1}{q}, q}(\unitint).
		\end{align}
	\end{lemma}
	\begin{proof}
		Let $m \in \mathbb{N}_0$, $b \in C^{\infty}_c(I)$, and $(s, q) \in \mathcal{A}_m$ be given. Let $\chi \in C^{\infty}_c(\mathbb{R})$ be any fixed smooth function satisfying $\chi \equiv 1$ on $[0, 1]$ and $\chi = 0$ on $\mathbb{R} \setminus [-1, 2]$ and let $\tilde{b}$ denote the zero extension of $b$ to $\mathbb{R}$. For $g \in C^{\infty}_c(\mathbb{R})$, define
		\begin{align*}
			\tilde{\mathcal{E}}_m(g)(x, y) = \chi(y) \frac{(-y)^m}{m!} \int_{\mathbb{R}}  \tilde{b}(t) g(x + ty) \ dt, \qquad (x, y) \in \mathbb{R}.
		\end{align*}
		Thanks to \cite[Lemma 4.2]{Arn88}, there holds
		\begin{align}
			\label{eq:proof:whole-space-extension-bound}
			\| \tilde{\mathcal{E}}_m(g) \|_{s, q, \mathbb{R}^2} \lesssim_{b, \chi, m, s, q} \|g\|_{s-m-\frac{1}{q}, q, \mathbb{R}} \qquad \forall g \in C^{\infty}_c(\mathbb{R}).
		\end{align}
		By density, \cref{eq:proof:whole-space-extension-bound} holds for all $g \in W^{s-m-\frac{1}{q}, q}(\mathbb{R})$.
		
		Let $f \in W^{s-\frac{1}{q}, q}(\unitint)$ and let $\tilde{f}$ denote an extension of $f$ to $\mathbb{R}$ satisfying \\
		$\|\tilde{f}\|_{s-\frac{1}{q}, q, \mathbb{R}} \lesssim_{s, q} \|f\|_{s - \frac{1}{q}, q, \unitint}$ and $\tilde{f}|_{\unitint} = f$, e.g. \cite{Devore93}. Applying \cref{eq:proof:whole-space-extension-bound} then gives
		\begin{align*}
			\| \mathcal{E}_m[b]( f ) \|_{s, q, T} = 	\| \tilde{\mathcal{E}}_m[\tilde{b}]( \tilde{f} ) \|_{s, q, T} \lesssim_{b, s, q} \| \tilde{f}\|_{s-m-\frac{1}{q}, q, \mathbb{R}} \lesssim_{s, q} \|f\|_{s-m-\frac{1}{q}, q, \unitint}.
		\end{align*}
	\end{proof}

	We are now in a position to prove \cref{lem:em3-derivative-interp-and-stability}.
	
	\begin{proof}[Proof of \cref{lem:em3-derivative-interp-and-stability}]
		For all $(s, q) \in \mathcal{A}_m$, \cref{eq:em-wsp-bound,eq:emq1-hm-stability} give
		\begin{alignat*}{2}
			\| \mathcal{E}_m^{[1]}[b](f) \|_{\beta, q, T} &\lesssim_{b, m, \beta, q} \|f\|_{q, \unitint, \beta q + 1} \qquad & &\forall f \in L^q(\unitint; t^{\beta q + 1}dt), \ 0 \leq \beta \leq m, \\
			\| \mathcal{E}_{m}^{[1]}[b](f)\|_{s, q, T} &\lesssim_{b,m,s,q} \|f\|_{s-m-\frac{1}{q}, q, \unitint} \qquad & &\forall f \in W^{s-m-\frac{1}{q}, q}(I).
		\end{alignat*}
	 	Additionally, for any $f \in C^{\infty}(\bar{\unitint})$, there holds
		\begin{align*}
			\partial_y^j \mathcal{E}_m^{[1]}(f)(x, y) = \sum_{i=0}^{j} \binom{j}{i} \frac{(-1)^{m+j} y^{m-i}}{(m-i)!} \int_{\unitint} b(t) t^{j-i} f^{(j-i)}(x + ty) \ dt, \qquad 0 \leq j \leq m,
		\end{align*}
		and so $\partial_y^j \mathcal{E}_m^{[1]}(f)(x, 0) = f(x) \delta_{jm}$ for $0 \leq x \leq 1$. Moreover, if $f \in \mathcal{P}_p(\unitint)$, $p \in \mathbb{N}_0$, then direct verification reveals that $\mathcal{E}_m^{[1]}(f) \in \mathcal{P}_{p+m}(T)$. 
		
		The result for $f \in W^{s, q}(\gamma_1)$ now follows from the smoothness of the map $\varphi_1$ \cref{eq:edge1-mapping}, while the result for $\mathcal{E}_{m}^{[i]}$, $i \in \{2, 3\}$, follows from the chain rule and the smoothness of the mappings $R$ and $R^{-1}$.
	\end{proof}
	
	\begin{remark}
		\label{rem:stability-without-int-condition}
		Note that the above proof shows that \cref{eq:em3-stability} holds without the restriction $\int_{I} b(t) \ dt = 1$.
	\end{remark}

	\subsection[Proof of \cref{lem:mmr3-derivative-interp}]{Proof of \cref{lem:mmr3-derivative-interp}}	
		For $k \in \mathbb{N}_0$ and $\beta \in [0, 1)$, define $W_L^{k+ \beta, q}(\unitint)$ by identifying $\gamma_1$ with $\unitint$. Let $(s, q) \in \mathcal{A}_m$. We first prove the following for $m, r \in \mathbb{N}_0$, $b \in C^{\infty}_c(\unitint)$, and $f \in W^{s-m-\frac{1}{q}, q}(I) \cap W_L^{\min\{ s-m, r\} - \frac{1}{q}, q}(I)$:
		\begin{align}
			\label{eq:proof:mmr1-stability}
			\|\mathcal{M}_{m, r}^{[1]}[b](f)\|_{\beta, q, T} \lesssim_{b, m, r, \beta, q} \begin{cases}
				\leftnorm{f}{\beta-m-\frac{1}{q}, q, \unitint} & \text{if }  \beta \leq m+r, \\
				\|f\|_{\beta-m-\frac{1}{q}, q, \unitint} & \text{if } \beta > {m+r},
			\end{cases}
		\end{align}		
		where ${m + 1 \leq \beta \leq s}$, $(\beta, q) \in \mathcal{A}_m$, and $W_L^{- \frac{1}{q}, q}(I) := L^q(I)$ for notational convenience. 
		
		We proceed by induction on $r$. The case $r = 0$ is a consequence of \cref{eq:em3-stability,rem:stability-without-int-condition}. Now assume that \cref{eq:proof:mmr1-stability} holds for some fixed $r \geq 0$ and all $m \in \mathbb{N}_0$ and $b \in C^{\infty}_c(\unitint)$. Let $f \in W^{s-m-\frac{1}{q}, q}(\unitint) \cap W_L^{\min\{s-m, r+1\} - \frac{1}{q}, q}(\unitint)$. The following identity will be useful: 
		\begin{align*}
			\mathcal{M}_{m, r}^{[1]}(f) - \mathcal{M}_{m, r+1}^{[1]}(f) &= x^r \frac{(-y)^m}{m!} \int_\unitint b(t) \frac{f(x + ty)}{(x+ty)^r} \left\{ 1 - \frac{x}{x+ ty} \right\} \ dt \\
			&= -x^r \frac{(-y)^{m+1}}{m!} \int_\unitint t b(t) \frac{f(x + ty)}{(x+ty)^{r+1}} \ dt \\
			&= -(m+1) \mathcal{M}_{m+1, r}[\tau b]\left(\tau^{-1} f \right)(x, y).
		\end{align*}
		First consider the case ${s = m+1}$. Thanks to \cref{eq:fovertkp12-bound}, there holds ${\tau^{-1} f \in L^q(\unitint; t dt)}$ and
		\begin{align*}
			 \| \tau^{-1} f\|_{q, \unitint, {1}} = \| {\tau^{-(1-\frac{1}{q})} f}\|_{q, \unitint} 
			\lesssim_{s, q} \leftnorm{f}{{1-\frac{1}{q}}, q, I}.
		\end{align*}
		Now applying \cref{eq:emq1-hm-stability} gives
		\begin{align}
			\label{eq:proof:small-s-bound}
			\| \mathcal{M}_{m+1, r}^{[1]}[\tau b](\tau^{-1} f) \|_{{m+1}, q, T} \lesssim_{b, m, r, q} \leftnorm{f}{{1-\frac{1}{q}}, q, I}.
 		\end{align}
		Now consider the case ${2 \leq s - m \leq r+1}$ for $r \geq 1$. Then, $\tau^{-1} f \in W_L^{s-m-1 - \frac{1}{q}}(\unitint)$ by \cref{eq:fovert-left-bound}, which combined with the inductive hypothesis gives
		\begin{align}
			\label{eq:proof:medium-s-bound}
			\| \mathcal{M}_{m+1, r}^{[1]}[\tau b](\tau^{-1} f) \|_{s, q, T} \lesssim_{b, m, r, s, q} \leftnorm{\tau^{-1} f}{s-m-1-\frac{1}{q}, q, \unitint} \lesssim_{s, q} \leftnorm{f}{s-m-\frac{1}{q}, q, \unitint}.
		\end{align}	
		Now let $s-m > r + 1$. By \cref{eq:fovert-bound,eq:fovert-left-bound} $\tau^{-1} f \in W^{s-m-1 - \frac{1}{q}}(\unitint) \cap W_L^{r - \frac{1}{q}}(\unitint)$, and the inductive hypothesis and \cref{eq:fovert-bound} give
		\begin{align*}
			\| \mathcal{M}_{m+1, r}^{[1]}[\tau b](\tau^{-1} f) \|_{s, q, T} \lesssim_{b, m, r, s, q} \|\tau^{-1} f\|_{s-m-1-\frac{1}{q}, q, \unitint} \lesssim_{s, q} \|f\|_{s-m-\frac{1}{q}, q, \unitint}.
		\end{align*}
		Thanks to the triangle inequality, we have shown that
		\begin{align*}
			\|\mathcal{M}_{m, r+1}^{[1]}[b](f)\|_{s, q, T} \lesssim_{b, m, r, s, q} \begin{cases}
				\leftnorm{f}{s-m-\frac{1}{q}, q, \unitint} & \text{if } {s = m+1,} \\
				 	\leftnorm{f}{s-m-\frac{1}{q}, q, \unitint} & \text{if } {2 \leq s-m \leq r+1}, \\
				\|f\|_{s-m-\frac{1}{q}, q, \unitint} & \text{if } s-m > r+1,
			\end{cases}
		\end{align*}	
		for any $b \in C^{\infty}_c(\unitint)$ and all $f \in W^{s-m-\frac{1}{q}, q}(\unitint) \cap W_L^{\min\{s-m, r+1\} - \frac{1}{q}, q}(\unitint)$, where $(s, q) \in \mathcal{A}_m$. 
		
		For the remaining case ${1 < s-m < 2}$, $(s, q) \in \mathcal{A}_m$, and $r \geq 1$, we apply an interpolation argument. More specifically,  $\mathcal{M}_{m, r+1}^{[1]}[b]$ maps  $W_L^{1 - \frac{1}{q}, q}(\unitint)$ into $W^{m+1, q}(T)$ and $W_L^{2 - \frac{1}{q}, q}(\unitint)$ into $W^{m+2, q}(T)$. Consequently,  
		\begin{align*}
			\mathcal{M}_{m, r+1}^{[1]}[b] : [W_L^{1 - \frac{1}{q}, q}(\unitint), W_L^{2 - \frac{1}{q}, q}(\unitint) ]_{\theta, q} \to [W^{m+1, q}(T), W^{m+2, q}(T)]_{\theta, q}
		\end{align*}	
		for any $0 \leq \theta \leq 1$. Choosing $\theta = s - m - 1$ and applying \cref{eq:left-space-interpolation,eq:proof:sobolev-interp-tri} gives that $\mathcal{M}^{[1]}_{m, r+1}[b]$ maps $W_L^{s-\frac{1}{q}, q}(\unitint)$ into $W^{m+s, q}(T)$. This completes the proof of \cref{eq:proof:mmr1-stability}. \Cref{eq:mmr3-stability} now follows from the smoothness of \cref{eq:edge1-mapping}. Direct computation then shows  \cref{eq:mmr3-derivative-interp}.
		
		Suppose further that $f \in \mathcal{P}_p(\gamma_1)$, $p \in \mathbb{N}_0$ with $\partial_{t}^{i} f(\bdd{a}_{2}) = 0$ for $0 \leq i \leq r-1$. Then, $t^{-r} f \circ \varphi(t) \in \mathcal{P}_{p-r}(\unitint)$, and so $\mathcal{M}_{m, r}^{[1]}(f) \in \mathcal{P}_{p+m}(T)$.
		
		The result for $\mathcal{M}_{m, r}^{[i]}$, $i \in \{2, 3\}$, follows from the chain rule and the smoothness of the mappings $R$ and $R^{-1}$. \hfill \proofbox

	 	\subsection[Proof of \cref{lem:smr3-derivative-interp}]{Proof of \cref{lem:smr3-derivative-interp}}	
	 	
	 	Let  $m \in \mathbb{N}_0$, $r \in \mathbb{N}$, $b \in C^{\infty}_c(\unitint)$, $(s, q) \in \mathcal{A}_m$ be as in the statement of the lemma. Let $\xi_i, \eta_i \in \mathcal{P}_{i-1}(\unitint)$, $i \in \{1, 2, \ldots, r\}$, be the components from the partial fraction decomposition of $\{ t(1-t) \}^{-r}$:
	 	\begin{align*}
	 		\{t(1-t)\}^{-r} = \sum_{i=1}^{r} \left\{  \xi_i(t) t^{-i} + \eta_i(t) (1-t)^{-i}\right\}.
	 	\end{align*}
	 	Then, for $f \in W^{s-\frac{1}{q}, q}(\gamma_1) \cap W^{\min\{s, r\} - \frac{1}{q}, q}_{00}(\gamma_1)$, there holds
	 	\begin{align*}
	 		\mathcal{S}_{m ,r}^{[1]}(f)(x, y) &= x^{r-i} (1-x-y)^r \sum_{i=1}^{r}  \mathcal{M}_{m, i}^{[1]}[b]( \xi_i f )(x, y) \\
	 		&\qquad + x^i(1-x-y)^{r-i} \sum_{i=1}^{r}  \mathcal{M}_{m, i}^{[1]}[\hat{b}] ( \hat{\eta}_i \hat{f}  )(1-x-y, y),
	 	\end{align*}
 		where $\hat{b}(t) = b(1-t)$ and $\hat{\eta}_i(t) = \eta(1-t)$ for $t \in \unitint$, while $\hat{f}(x,y) = f(1-x, y)$ for $(x,y) \in \gamma_1$. Since $f \in W^{\min\{s, r\} - \frac{1}{q}, q}_{00}(\gamma_1)$, $f \in W^{\min\{s, r\} - \frac{1}{q}, q}_{L}(\gamma_1)$ and $\hat{f} \in W^{\min\{s, r\} - \frac{1}{q}, q}_{L}(\gamma_1)$, and so 
	 	\begin{align*}
	 		\|\mathcal{S}_{m ,r}^{[1]}(f)\|_{\beta, q, T} \lesssim_{m, r, \beta, q}  \sum_{i=1}^{r} \left\{ \| \mathcal{M}_{m, i}^{[1]}[b](\xi_i f) \|_{\beta, q, T} + \|\mathcal{M}_{m, i}^{[1]}[b] (\hat{\eta}_i \hat{f}) \|_{\beta, q, T} \right\}
	 	\end{align*}
	 	for $0 \leq \beta \leq s$. \Cref{eq:smr3-stability-low,eq:smr3-stability-high} now follow from the triangle inequality and \cref{eq:mmr3-stability} on noting that
	 	\begin{align*}
	 		\leftnorm{\xi_i f}{\beta-m-\frac{1}{q}, q, \unitint} + \leftnorm{\hat{\eta}_i \hat{f} }{\beta-m-\frac{1}{q}, q, \unitint} \lesssim_{m, r, \beta, q} \zznorm{f}{\beta-m-\frac{1}{q}, q, \unitint}
	 	\end{align*}
 		for ${m + 1 \leq \beta \leq m+r}$, $(\beta, q) \in \mathcal{A}_m$. Direct computation then gives \cref{eq:smr3-derivative-interp}.
	 	
	 	Suppose further that $f \in \mathcal{P}_p(\gamma_1)$, $p \in \mathbb{N}_0$ with $\partial_t^i f|_{\partial \gamma_1} = 0$ for $i \in \{0, 1, \ldots, r-1\}$. Then, $(d_2 d_3)^{-r} f \in \mathcal{P}_{p-r}(\gamma_1)$, and so $\mathcal{S}_{m, r}^{[1]}(f) \in \mathcal{P}_{p+m}(T)$.	
	 	
	 	The result for $\mathcal{S}_{m, r}^{[i]}$, $i \in \{2, 3\}$, follows from the chain rule and the smoothness of the mappings $R$ and $R^{-1}$. \hfill \proofbox
	
	\section{Generalization to arbitrary order normal derivatives}
	\label{sec:generalization}
	
	In \cref{sec:constructing-lifting}, we constructed an operator $\tilde{\mathcal{L}}$ that boundedly lifts a pair of functions defined on the boundary $\partial T$ to a single function defined on the whole triangle $T$. We now consider the generalized problem of boundedly lifting $m+1$ functions on $\partial T$ to one function on $T$. To make this statement precise, we first review the regularity of the traces of $u \in W^{s, q}(T)$, for $(s, q) \in \mathcal{A}_m$, as we did in \cref{sec:trace-review} for $(s, q) \in \mathcal{A}_1$. To this end, we define the $m$th-order trace operator $\sigma^m$, $m \in \mathbb{N}$, edge by edge according to the rule 
	\begin{multline*}
		\sigma_i^m(f^0, f^1, \ldots, f^m)_{j_1 j_2 \ldots, j_m} := \partial_t 	\sigma_i^{m-1}(f^0, f^1, \ldots, f^{m-1})_{j_1 j_2 \ldots, j_{m-1}} \cdot (\unitvec{t}_l)_{j_m} \\ + f^m \cdot (\unitvec{n}_l)_{j_1} (\unitvec{n}_l)_{j_2} \cdots (\unitvec{n}_l)_{j_m},
 	\end{multline*}
 	where $\sigma^0$ is defined in \cref{eq:sigma0-def}. Note that the above definition coincides with \cref{eq:sigma1-def} when $m=1$.

	Let $F = (u|_{\partial T}, \partial_n u|_{\partial T}, \ldots, \partial_n^m u|_{\partial T})$. Applying the same arguments as in \cref{sec:trace-review}, we see that $\sigma^m(F) = D^m u$ on $\partial T$, and so  \cref{eq:review:usual-trace-conditions} gives the edge regularity condition \cref{eq:sigmam-edge-reg}. Similarly, we obtain continuity conditions of $\sigma^m(F)$ from \cref{eq:review:usual-trace-conditions} and \cref{eq:review:sobolev-embed-cont} for particular values of $s$ and $q$ as stated in \cref{eq:sigmam-cont} with $l=0$. By forming mixed derivatives at a vertex using tangential derivatives of $\sigma^m(F)$, as we did with $\sigma^1$ in \cref{sec:trace-review}, we obtain additional conditions which we now describe. For a $d$-dimensional tensor $S$ and $v \in \mathbb{R}^2$, we define
	\begin{align*}
		{v}^{\otimes 0} \cdot S  = S \quad \text{and} \quad	{v}^{\otimes l} \cdot S  = S_{i_1 i_2 \ldots i_l} {v}_{i_1} {v}_{i_2} \cdots {v}_{i_l}, \qquad l \in \{1, 2, \ldots, d\}.
	\end{align*}
	Then, using the symmetry of the derivative tensors, we have
	\begin{multline*}
		 \unitvec{t}_{i+2}^{\otimes l} \cdot \partial_t^l \sigma_{i+1}^m(F)(\bdd{a}_{i})
		 = \unitvec{t}_{i+2}^{\otimes l} \cdot \partial_{t_{i+1}}^l D^m u(\bdd{a}_i) = \unitvec{t}_{i+2}^{\otimes l} \cdot \left( \unitvec{t}_{i+1}^{\otimes l} \cdot D^{m+l} u(\bdd{a}_i) \right)  \\
		= \unitvec{t}_{i+1}^{\otimes l} \cdot \left( \unitvec{t}_{i+2}^{\otimes l} \cdot D^{m+l} u(\bdd{a}_i) \right)  = \unitvec{t}_{i+1}^{\otimes l} \cdot \partial_{t_{i+2}}^l D^m u(\bdd{a}_i) =  \unitvec{t}_{i+1}^{\otimes l} \cdot \partial_t^l \sigma_{i+2}^m(F)(\bdd{a}_{i})
	\end{multline*}
	for $l \in \{0, 1, \ldots, m\}$. Thus, we obtain at most $m$ additional continuity conditions at the vertices as stated in \cref{eq:sigmam-cont-combo}. In summary, $\sigma_i^m(F)$ satisfies the following conditions:
	\begin{enumerate}
		
		\item $W^{s-m-\frac{1}{q}, q}$ regularity on each edge:
		\begin{align}
			\label{eq:sigmam-edge-reg}
			\sigma_i^m(F) \in W^{s-m-\frac{1}{q}, q}(\gamma_i), \qquad i \in \{1,2,3\}.
		\end{align}
		
		\item Continuity at vertices:
		\begin{subequations}
			\label{eq:sigmam-cont-combo}
			\begin{alignat}{2}
				\label{eq:sigmam-cont}
				\unitvec{t}_{i+2}^{\otimes l} \cdot \partial_t^l \sigma_{i+1}^m(F)(\bdd{a}_{i})   =  \unitvec{t}_{i+1}^{\otimes l} \cdot \partial_t^l \sigma_{i+2}^m(F)(\bdd{a}_{i}) & \qquad & &\text{if } (s-m-l)q > 2, \\
				\label{eq:sigmam-cont-int}
				\mathcal{I}_{i}^{q}(\unitvec{t}_{i+2}^{\otimes l} \cdot \partial_t^l \sigma_{i+1}^m(F), \unitvec{t}_{i+1}^{\otimes l} \cdot \partial_t^l \sigma_{i+2}^m(F)  ) < \infty& \qquad & & \text{if } (s-m-l)q = 2,
			\end{alignat}
		\end{subequations}
		for  $l \in \{0, 1, \ldots, m\}$ and $i \in \{1,2,3\}$.
	\end{enumerate}

	Motivated by the above conditions, we define the space $X_m^{s,q}(\partial T)$, for $(s, q) \in \mathcal{A}_m$ as follows:
	\begin{multline}
		\label{eq:xsqm-def}
		X_m^{s,q}(\partial T) := \{ (f^0, f^1, \ldots, f^m) \in L^q(T)^{m+1} :  \text{$\sigma^k(f^0, f^1, \ldots, f^k)$ satisfies} \\ \text{ \cref{eq:sigmam-cont-combo,eq:sigmam-edge-reg} for $k \in \{0, 1, \ldots, m\}$} \},
	\end{multline}
	equipped with the norm
	\begin{multline*}
		\| (f^0, f^1, \ldots f^m) \|_{X_m^{s,q}, \partial T}^q := \sum_{i=1}^{3} \sum_{k=0}^{m}  \| f_i^k \|_{s-k-\frac{1}{q}, q, \gamma_i}^q  \\
		+ \sum_{i=1}^{3}  \begin{cases}
				\mathcal{I}_{i}^{q}(\unitvec{t}_{i+2}^{\otimes l} \cdot \partial_t^l \sigma_{i+1}^m(F), \unitvec{t}_{i+1}^{\otimes l} \cdot \partial_t^l \sigma_{i+2}^m(F)) & \text{if } (s-m-l)q = 2, \\
			0 & \text{otherwise}.
		\end{cases}
	\end{multline*}
	Note that with the above definition, $X_1^{s, q}(\partial T) = X^{s, q}(\partial T)$, where $X^{s, q}(\partial T)$ is defined in \cref{eq:trace-1-space-def}. The above discussion leads to the following trace estimate.
	\begin{lemma}
		For every $m \in \mathbb{N}_0$, $(s, q) \in \mathcal{A}_m$, and $u \in W^{s, q}(T)$, the traces satisfy \newline $(u|_{\partial T}, \partial_n u|_{\partial T}, \ldots, \partial_n^m u|_{\partial T}) \in X_m^{s,q}(\partial T)$ and
		\begin{align}
			\label{eq:trace-thm-gen}
			\|(u, \partial_n u, \ldots, \partial_n^m u)\|_{X_m^{s,q}, \partial T} \lesssim_{m, s, q} \|u\|_{s, q, T}.
		\end{align}
	\end{lemma}
	\noindent The remainder of this section is devoted to proving the following generalization of \cref{thm:tilde-l1-existence}.
	\begin{theorem}
		\label{thm:tilde-lm-existence}
		Let $m \in \mathbb{N}_0$. There exists a single linear operator
		\begin{align*}
			\tilde{\mathcal{L}}_m : \bigcup_{ (s, q) \in \mathcal{A}_m } X_m^{s,q}(\partial T) \to W^{m, 1}(T)
		\end{align*}
		satisfying the following properties. For all $(s, q) \in \mathcal{A}_m$ and  $F = (f^0, f^1, \ldots, f^m) \in X_m^{s,q}(\partial T)$, $\tilde{\mathcal{L}}_m(F) \in W^{s,q}(T)$ and there holds
		\begin{align}
			\label{eq:tilde-lm-prop}
			\partial_n^k \tilde{\mathcal{L}}_m (F)|_{\partial T} = f^k, \quad k \in \{0, 1, \ldots, m\}, \quad \text{and} \quad \|\tilde{\mathcal{L}}_m (F)\|_{s, q, T} \lesssim_{m, s, q} \|F\|_{X_m^{s,q}, \partial T}.
		\end{align}
		Moreover, if for some $p \in \mathbb{N}_0$ and all $i \in \{1,2,3\}$, there holds
		\begin{subequations}
			\label{eq:gen-poly-compat-conditions}
			\begin{alignat}{2}
				f_i^k &\in \mathcal{P}_{p-k}(\gamma_i) \qquad & &k \in \{0, 1, \ldots, m\}, \\
				\sigma_{i+1}^k(f^0, f^1, \ldots, f^k)(\bdd{a}_i) &= \sigma_{i+2}^k(f^0, f^1, \ldots, f^k)(\bdd{a}_i) \qquad & &k \in \{0, 1, \ldots, m\}, \\
				\unitvec{t}_{i+2}^{\otimes l} \cdot \partial_t^l \sigma_{i+1}^m(F)(\bdd{a}_{i})   &=  \unitvec{t}_{i+1}^{\otimes l} \cdot \partial_t^l \sigma_{i+2}^m(F)(\bdd{a}_{i}) \qquad & &l \in \{1, 2, \ldots, m\},
			\end{alignat}
		\end{subequations}
		 then $\tilde{\mathcal{L}}_m(F) \in \mathcal{P}_{p}(T)$ and \cref{eq:tilde-lm-prop} holds for all $(s, q) \in \mathcal{A}_m$.
	\end{theorem}
	
	\subsection{Two technical lemmas}
	
	We first generalize \cref{lem:trace-diff-left-space}.
	\begin{lemma}
		\label{lem:trace-diff-left-space-gen}
		Let $m \in \mathbb{N}_0$, $(s, q) \in \mathcal{A}_m$, and $F = (f^0, f^1, \ldots, f^m) \in X_m^{s, q}(\partial T)$. Suppose that for some $l \in \{0, 1, \ldots, m\}$, there holds
		\begin{enumerate}
			\item[(i)] $f_{1}^0 = f_{1}^1 = \cdots = f_1^m = 0$, and
			
			\item[(ii)] $f_{2}^0 = f_{2}^1 = \cdots = f_{2}^{l-1} = 0$ if $l \geq 1$.
		\end{enumerate}
		Then, $f_{2}^l \in W_L^{\beta - \frac{1}{q}, q}(\gamma_{2})$ with $\beta = \min\{ s-l, m \}$, and there holds
		\begin{align}
			\label{eq:trace-diff-left-space-bound-gen}
			\leftnorm{f_{2}^{l}}{\beta-\frac{1}{q}, q, \gamma_{2}} &\lesssim_{\beta, q}   \| F \|_{X_m^{s, q}, \partial T}.
		\end{align}
		If, in addition, $F$ satisfies \cref{eq:gen-poly-compat-conditions}, then $\partial_t^{j} f^l_{2}(\bdd{a}_{3}) = 0$, $j \in \{0, 1, \ldots, m\}$.
	\end{lemma}
	\begin{proof}
		Let $m \in \mathbb{N}_0$, $l \in \{0, 1, \ldots, m\}$, $(s, q) \in \mathcal{A}_m$, and $F = (f^0, f^1, \ldots, f^m) \in X_m^{s, q}(\partial T)$ be as in the statement of the lemma. By definition, there holds
		\begin{align}
			\label{eq:proof:trace-comp-triangle-ineq-gen}
			\|f_{j+1}^{l} \|_{\beta-\frac{1}{q}, q, \gamma_{j+1}} &\lesssim_{m, \beta, q}   \| F \|_{X_m^{s, q}, \partial T},
		\end{align}
		and so it remains to verify the conditions \cref{eq:left-space-conditions} and bound the weighted $L^q$ norm term in \cref{eq:left-norm-def-gamma} when $s - 2/q \in \mathbb{Z}$.

		Let $\phi_{1}(h) = \bdd{a}_{3} - h \unitvec{t}_{1}$, and $\phi_{2}(h) = \bdd{a}_{3} + h \unitvec{t}_{2}$ for $0 \leq h \leq 1$ be the same edge parametrizations as in the proof of \cref{lem:trace-diff-left-space}. Thanks to the identity
		\begin{align}
			\label{eq:proof:sigma-switch-tangent-deriv}
			\partial_t^r \sigma_2^l(F) = \unitvec{t}_2^{\otimes r} \cdot \sigma_2^{l+r}(F), \qquad r \in \{0, 1, \ldots, m-l\},
		\end{align}
		where $\sigma^{j}(F) = \sigma^j(f^0, f^1, \ldots, f^{j})$, we obtain the following for $k \in \{0, 1, \ldots, m-l-1\}$:
		\begin{align*}
			\partial_h^k \{ f_2^l \circ \phi_2 \} &= \partial_h^k \left\{  \unitvec{n}_2^{\otimes l} \cdot \sigma_2^l (F) \circ \phi_2 \right\} \\
			&= \unitvec{t}_2^{\otimes k} \cdot \unitvec{n}_2^{\otimes l} \cdot \sigma_2^{l+k}(F) \circ \phi_2 \\
			&= \unitvec{t}_2^{\otimes l+k} \cdot \unitvec{n}_2^{\otimes l} \cdot \{ \sigma_2^{l+k}(F) \circ \phi_2 - \sigma_1^{l+k}(F) \circ \phi_1 \},
		\end{align*}
 		where we used (i) in the final step. \Cref{eq:sigmam-cont} then gives $\partial_t^k f^l_2(\bdd{a}_3) = 0$.
 		
		For $k \in \{m-l, m-l+1, \ldots, m\}$ and $k < s - l + 1/q$, there holds
		\begin{align*}
			\partial_t^k f_2^l  &= \unitvec{n}_2^{\otimes l} \cdot \unitvec{t}_2^{\otimes m-l} \cdot \partial_t^{k-m+l} \sigma_2^{m}(F). 
		\end{align*}
		Since $\unitvec{t}_1$ and $\unitvec{t}_{2}$ are linearly independent, there exist constants $c_1, c_2 \in \mathbb{R}$ such that  $\unitvec{n}_{2} = c_1 \unitvec{t}_1 + c_2 \unitvec{t}_{2}$, and so \cref{eq:proof:sigma-switch-tangent-deriv} gives
		\begin{align*}
			\partial_t^k f_2^l  &= \sum_{i=0}^{l} \tilde{c}_i \unitvec{t}_1^{\otimes i} \cdot \unitvec{t}_2^{\otimes m-i} \cdot \partial_t^{k-m+l} \sigma_2^{m}(F) 
				= \sum_{i=0}^{l} \tilde{c}_i \unitvec{t}_1^{\otimes i}  \cdot \partial_t^{k+l-i} \sigma_2^{i}(F) 
		\end{align*}
		for suitable constants $\{\tilde{c}_i\}$. Thanks to (ii), $\partial_h^{k+l-i} \{\sigma_2^i(F) \circ \phi_2\} = 0$ for $i \in \{0, 1, \ldots, l-1\}$, and so
		\begin{align*}
			\partial_h^k \{ f_2^l \circ \phi_2 \}
			&= \tilde{c}_l \unitvec{t}_1^{\otimes l} \cdot \unitvec{t}_2^{\otimes m-l}  \cdot \partial_h^{k-m+l} \{\sigma_2^{m}(F) \circ \phi_2\} \\
			&= \tilde{c}_l \unitvec{t}_2^{\otimes m-l} \cdot \unitvec{t}_1^{\otimes m-k} \cdot \{ \unitvec{t}_1^{\otimes k-m+l}  \cdot \partial_h^{k-m+l} \{\sigma_2^{m}(F) \circ \phi_2\} \\
				&\qquad \qquad \qquad \qquad \qquad - \unitvec{t}_2^{\otimes k-m+l}  \cdot \partial_h^{k-m+l} \{\sigma_1^{m}(F) \circ \phi_1\} \}, 
		\end{align*}
		where we used (i). Using \cref{eq:sigmam-cont} and combining with the case $k \leq m-l-1$ then gives 
		\begin{align*}
			\partial_t^k f_2^l(\bdd{a}_3) &= 0 \qquad \text{if } (s-k-l)q > 2  \text{ and }  0 \leq k \leq m. 
		\end{align*}
		When $(s-k-l)q = 2$, we have the bound 
		\begin{align*}
			\|d_3^{-\frac{1}{q}} \partial_t^k f^l\|_{q, \gamma_2}^q
			\lesssim_q \|\partial_t^k f^l\|_{q, \gamma_{2}}^q + \mathcal{I}_{3}^q( \unitvec{t}_2^{\otimes j} \cdot \partial_t^{j} \sigma_{1}^m(F), \unitvec{t}_1^{\otimes j} \cdot \partial_t^{j} \sigma_{2}^m(F) ),
		\end{align*}
		where $j = k-m+l$.
		Collecting results then gives $f_{2}^l \in W_L^{\beta - \frac{1}{q}, q}(\gamma_{2})$ and \cref{eq:trace-diff-left-space-bound-gen}.
		
		Now suppose that $F$ satisfies \cref{eq:gen-poly-compat-conditions}. Then, we have already shown that $f^l_{2} \in W_L^{m - \frac{1}{q},q}(\gamma_{2})$ for all $1 < q < \infty$, and so $\partial_t^j f^l_{2} (\bdd{a}_{3}) = 0$, $j \in \{0, 1, \ldots, m\}$.	
	\end{proof}
	We also have the following generalization of \cref{lem:trace-diff-zz-space}.
		\begin{lemma}
		\label{lem:trace-diff-zz-space-gen}
		Let $(s, q) \in \mathcal{A}_m$, and $F = (f^0, f^1, \ldots, f^m) \in X_m^{s, q}(\partial T)$. Suppose that for some $l \in \{0, 1, \ldots, m\}$ there holds
		\begin{enumerate}
			\item[(i)] $f_{i}^{0} = f_i^1 = \cdots = f_i^m =0$ for $i \in \{1, 2\}$, and
			
			\item[(ii)] $f_{3}^{0} = f_3^1 = \cdots = f^{l-1}_3 = 0$ if $l \geq 1$.
		\end{enumerate}
		Then, $f_{3}^{n} \in W_{00}^{\beta - \frac{1}{q}, q}(\gamma_{3})$ with $\beta = \min\{ s-l, m \}$, and there holds
		\begin{align}
			\label{eq:trace-diff-zz-space-bound-gen}
			\zznorm{f_{3}^{l}}{\beta-\frac{1}{q}, q, \gamma_{3}} &\lesssim_{\beta, q}   \| F \|_{X_m^{s, q},\partial T}.
		\end{align}
		If, in addition, $F$ satisfies \cref{eq:gen-poly-compat-conditions}, then then $\partial_t^j f^l_{3}|_{\partial \gamma_{3}} = 0$, $j \in \{0, 1, \ldots, m\}$.
	\end{lemma}
	\begin{proof}
		Let $m \in \mathbb{N}_0$, $l \in \{0, 1, \ldots, m\}$, $(s, q) \in \mathcal{A}_m$, and $F = (f^0, f^1, \ldots, f^m) \in X_m^{s, q}(\partial T)$ be as in the statement of the lemma. Applying the same arguments as in the proof of \cref{lem:trace-diff-left-space-gen}, replacing $\gamma_1$ and $\gamma_2$ with $\gamma_2$ and $\gamma_3$ gives $f_{3}^{l} \in W_{L}^{\beta - \frac{1}{q}, q}(\gamma_{3})$ with 
		\begin{align}
			\label{eq:proof:gamma3-left-norm-bound-gen}
			\leftnorm{f_{3}^{l}}{\beta-\frac{1}{q}, q, \gamma_{3}} &\lesssim_{s, q}   \| F \|_{X_m^{s, q}, \partial T}.
		\end{align}
		Again applying the same arguments as in the proof of \cref{lem:trace-diff-left-space-gen} but reversing the roles of $\gamma_1$ and $\gamma_{2}$ and then replacing  $\gamma_1$ and $\gamma_2$ with $\gamma_2$ and $\gamma_3$ gives
		\begin{align*}
			\partial_t^k f_{3}^{l} (\bdd{a}_{2}) = 0 \qquad \text{for } 0 \leq k <  \min\left\{ s - k - \frac{2}{q}, m \right\}, 
		\end{align*}
		and if  $(s-k-l)q = 2$ for some $k \in \{0, 1, \ldots, m\}$, then
		\begin{align}
			\label{eq:proof:gamma3-right-norm-bound-gen}
			\| d_2^{-\frac{1}{q}} \partial_t^k f^l\|_{q, \gamma_3} \lesssim_{\beta, q} \|F\|_{X_m^{\beta, q}, \partial T}.
		\end{align}
		The inclusion $f_{3}^{l} \in W_{00}^{\beta - \frac{1}{q}, q}(\gamma_{3})$ then follows from \cref{eq:zz-space-conditions} on noting that $d_1 + d_2 \lesssim d_1 d_2$, which, in conjunction with \cref{eq:proof:gamma3-left-norm-bound-gen}, \cref{eq:proof:gamma3-right-norm-bound-gen}, and the triangle inequality, gives \cref{eq:trace-diff-zz-space-bound-gen}.
		
		Now suppose that $F$ satisfies \cref{eq:gen-poly-compat-conditions}. Then, we have already shown that $f^l_{3} \in W_{00}^{m - \frac{1}{q},q}(\gamma_3)$ for all $1 < q < \infty$, and so $\partial_t^j f^l_{3} (\bdd{a})|_{\gamma_3} = 0$, $j \in \{0, 1, \ldots, m\}$.
	\end{proof}

	\subsection{Construction of the lifting operator}
	\label{sec:gen-lifting-constr}
	
	We now extend the construction in \cref{sec:constructing-lifting}. Let $m \in \mathbb{N}_0$ be given and $b \in C^{\infty}_c(\unitint)$ with $\int_{\unitint} b(t) \ dt = 1$. For $F = (f^0, f^1, \ldots, f^m) \in L^q(\partial T)^{m+1}$, we formally define the following operators:
	\begin{subequations}
		\begin{alignat}{2}
			\mathcal{K}_{0}^{[1]}(F) &:= \mathcal{E}_0^{[1]}[b](f_1^0), \qquad & & \\
			\mathcal{K}_{i}^{[1]}(F) &:= \mathcal{K}_{i-1}^{[1]}(F) +  \mathcal{E}_i^{[1]}[b]\left( f_1^i - \partial_n^i \mathcal{K}_{ i-1}^{[1]}(F)|_{\gamma_1} \right), \qquad & & i \in \{1,2,\ldots,m\}, \\
	 		\mathcal{K}_{0}^{[2]}(F) &:= \mathcal{K}_{m}^{[1]}(F) + \mathcal{M}_{0, m}^{[2]}[b](f_2^0 - \mathcal{K}_{m}^{[1]}(F)|_{\gamma_2} ), \qquad & & \\
			\mathcal{K}_{i}^{[2]}(F) &:= \mathcal{K}_{i-1}^{[2]}(F) + \mathcal{M}_{i, m}^{[2]}[b](f_2^i - \partial_n^i \mathcal{K}_{i-1}^{[2]}(F)|_{\gamma_2} ), \qquad & & i \in \{1,2,\ldots,m\}, \\
			\mathcal{K}_{0}^{[3]}(F) &:= \mathcal{K}_{m}^{[2]}(F) + \mathcal{S}_{0, m}^{[3]}[b](f_3^0 - \mathcal{K}_{m}^{[2]}(F)|_{\gamma_3} ), \qquad & & \\
			\mathcal{K}_{i}^{[3]}(F) &:= \mathcal{K}_{i-1}^{[3]}(F) + \mathcal{S}_{i, m}^{[3]}[b](f_3^i - \partial_n^i \mathcal{K}_{i-1}^{[3]}(F)|_{\gamma_3} ), \qquad & & i \in \{1,2,\ldots,m\}, \\
			\label{eq:tilde-lm-def}
			\tilde{\mathcal{L}}_m(F) &:= \mathcal{K}_{m}^{[3]}(F). \qquad & & 
		\end{alignat}
	\end{subequations}
	We now prove \cref{thm:tilde-lm-existence}.
	
	\begin{proof}[Proof of \cref{thm:tilde-lm-existence}]
		Let $m \in \mathbb{N}_0$, $(s, q) \in \mathcal{A}_m$, and $F = (f^0, f^1, \ldots, f^m) \in X_m^{s, q}(\partial T)$. $\mathcal{K}_{0}^{[1]}$ is well-defined by \cref{lem:em3-derivative-interp-and-stability}, and arguing inductively shows that $\mathcal{K}_{i}^{[1]}$ is well-defined for $i \in \{0, 1, \ldots, m\}$. Repeatedly applying \cref{eq:em3-stability}, the triangle inequality, and the trace estimate \cref{eq:trace-thm-gen} gives
		\begin{align*}
			\|\mathcal{K}_{m}^{[1]}(F)\|_{s, q, T} &\lesssim_{m, s, q} \| \mathcal{K}_{m-1}^{[1]}(F) \|_{s, q, T} + \| f_1^m \|_{s-m-\frac{1}{q}, q, \partial T} \\
				&\lesssim_{m, s, q} \| \mathcal{K}_{m-2}^{[1]}(F) \|_{s, q, T} + \sum_{i=m-1}^{m} \| f_1^i \|_{s-i-\frac{1}{q}, q, \partial T} \\
				\cdots &\lesssim_{m, s ,q} \| \mathcal{K}_{0}^{[1]} \|_{s, q, T} + \sum_{i=1}^{m} \| f_1^i \|_{s-i-\frac{1}{q}, q, \partial T} \\
				&\lesssim_{m, s, q} \|F\|_{X^{s, q}_m, \partial T}.
		\end{align*}
		Moreover, \cref{eq:em3-derivative-interp} shows that $\partial_n^k \mathcal{K}_{m}^{[1]}(F)|_{\gamma_1} = f_1^k$, for $k \in \{0, 1, \ldots, m\}$.
		
		We now turn to $\mathcal{K}_{0}^{[2]}$. Applying \cref{lem:trace-diff-left-space-gen} gives $f_2^0 - \mathcal{K}_{m}^{[1]}(F)|_{\gamma_2} \in W^{s-\frac{1}{q}, q}(\gamma_2) \cap W_L^{\min\{s, m\} - \frac{1}{q}, q}(\gamma_2)$. Thus, $\mathcal{K}_{0}^{[2]}$ is well-defined by \cref{lem:mmr3-derivative-interp} with $\mathcal{K}_{0}^{[2]}(F)|_{\gamma_2} = f^0_2$, $\partial_n^k \mathcal{K}_{0}^{[2]}(F)|_{\gamma_1} = f_0^k$, for $k \in \{0, 1, \ldots, m\}$, and
		\begin{align*}
			\|\mathcal{K}_{0}^{[2]}(F)\|_{s, q, T} &\lesssim_{m, s, q} \|\mathcal{K}_{m}^{[1]}(F)\|_{s, q, T} + \leftnorm{f_2^0 - \mathcal{K}_{m}^{[1]}(F)}{s-\frac{1}{q}, q, \gamma_2} 
				\lesssim_{m, s, q} \|F\|_{X^{s, q}_m, \partial T}
		\end{align*}
		by \cref{eq:mmr3-stability,eq:trace-thm-gen,eq:trace-diff-left-space-bound-gen}. Arguing inductively by applying \cref{lem:mmr3-derivative-interp} and \cref{lem:trace-diff-left-space-gen} repeatedly shows that $\mathcal{K}_{i}^{[2]}$, $i \in \{0, 1, \ldots, m\}$, is well-defined with
		\begin{align*}
		\partial_n^k \mathcal{K}_{m}^{[2]}(F)|_{\gamma_1 \cup \gamma_2} = f^k, \quad k \in \{0, 1, \ldots, m\}, \ \ \text{and} \ \	\|\mathcal{K}_{m}^{[2]}(F)\|_{s, q, T} \lesssim_{m, s, q} \|F\|_{X^{s, q}_m, \partial T}.
		\end{align*}
	
		Next, we turn to $\mathcal{K}_{0}^{[3]}$. Applying \cref{lem:trace-diff-zz-space-gen} gives $f_3^0 - \mathcal{K}_{m}^{[2]}(F)|_{\gamma_3} \in W^{s-\frac{1}{q}, q}(\gamma_3) \cap W_{00}^{\min\{s, m\} - \frac{1}{q}, q}(\gamma_3)$. Thus, $\mathcal{K}_{0}^{[3]}$ is well-defined by \cref{lem:smr3-derivative-interp} with $\mathcal{K}_{0}^{[3]}(F)|_{\gamma_3} = f^0_3$, $\partial_n^k \mathcal{K}_{0}^{[3]}(F)|_{\gamma_1 \cup \gamma_2} = f_0^k$, for $k \in \{0, 1, \ldots, m\}$, and
		\begin{align*}
			\|\mathcal{K}_{0}^{[3]}(F)\|_{s, q, T} &\lesssim_{m, s, q} \|\mathcal{K}_{m}^{[2]}(F)\|_{s, q, T} + \zznorm{f_3^0 - \mathcal{K}_{m}^{[2]}(F)}{s-\frac{1}{q}, q, \gamma_2} 
			\lesssim_{m, s, q} \|F\|_{X^{s, q}_m, \partial T}
		\end{align*}
		by \cref{eq:trace-thm-gen,eq:trace-diff-zz-space-bound-gen,eq:smr3-stability-high}. Arguing inductively and analogously as above with \cref{lem:smr3-derivative-interp} and \cref{lem:trace-diff-zz-space-gen} repeatedly shows that $\mathcal{K}_{i}^{[2]}$, $i \in \{0, 1, \ldots, m\}$, is well-defined with
		\begin{align*}
			\partial_n^k \mathcal{K}_{m}^{[3]}(F)|_{\partial T} = f^k, \quad k \in \{0, 1, \ldots, m\}, \ \ \text{and} \ \	\|\mathcal{K}_{m}^{[3]}(F)\|_{s, q, T} \lesssim_{m, s, q} \|F\|_{X^{s, q}_m, \partial T}.
		\end{align*}
		This completes the proof of \cref{eq:tilde-lm-prop}.
		
		Finally, we assume that $F$ satisfies \cref{eq:gen-poly-compat-conditions}. By definition \cref{eq:xsqm-def}, $F \in X^{s, q}_m(\partial T)$. \Cref{lem:em3-derivative-interp-and-stability} then gives that $\mathcal{K}_{m}^{[1]}(F) \in \mathcal{P}_p(T)$. Repeatedly applying \cref{lem:trace-diff-left-space-gen} and \cref{lem:mmr3-derivative-interp} and arguing analogously as in the proof of \cref{thm:tilde-l1-existence} show that $\mathcal{K}_{m}^{[2]}(F) \in \mathcal{P}_{p}(T)$. Similar arguments based on \cref{lem:trace-diff-zz-space-gen} and \cref{lem:smr3-derivative-interp} then show that $\mathcal{K}_{m}^{[3]}(F) \in \mathcal{P}_{p}(T)$, which completes the proof.
	\end{proof}
	
	\section{Summary and future work}
	
	We have constructed a right inverse of the trace operator $u \mapsto (u|_{\partial T}, \partial_n u|_{\partial T}, \ldots, \partial_n^m u|_{\partial T})$, $m \in \mathbb{N}_0$, that maps suitable piecewise polynomial data on $\partial T$ into polynomials of the same degree and is bounded from $X_m^{s, q}(\partial T)$ into $W^{s, q}(T)$ for all $(s, q) \in \mathcal{A}_m$. One open problem is whether the above construction is also stable from the appropriate Besov space into $W^{s, q}(T)$ when $s-1/q \in \mathbb{Z}$ and $q \neq 2$ or from the trace of $W^{s, q}(T)$ with $m + 1/q < s < m + 1$ into $W^{s, q}(T)$, which arises in the analysis of high order discretizations of fractional PDEs. Another open problem is how to generalize the above construction to three or more space dimensions.
	
	\appendix
	 
	\section{Auxiliary 1D results}
	\label{sec:aux-1d}	

	\begin{lemma}
		Define the operator $\mathcal{H}_L$ formally by the rule 
		\begin{align*}
			\mathcal{H}_L f(t) = t^{-1} \int_{0}^{t} f(s) \ ds = \int_{0}^{1} f(ts) \ ds, \qquad t \in \unitint.
		\end{align*}
		For any real numbers $s \geq 0$ and $1 < q < \infty$, $\mathcal{H}_L$ is a bounded map of $W^{s, q}(\unitint)$ into $W^{s, q}(\unitint)$ and of $W^{s}_L(\unitint)$ into $W^{s}_L(\unitint)$. In particular,
		\begin{alignat}{2}
			\label{eq:hardy-average-bound}
			\| \mathcal{H}_L f \|_{s, q, \unitint} &\lesssim_{s, q} \| f\|_{s, q, \unitint} \qquad & &\forall f \in W^{s, q}(\unitint), \\
			\label{eq:hardy-average-left-bound}
			\leftnorm{\mathcal{H}_L f}{s, q, \unitint} &\lesssim_{s, q} \leftnorm{f}{s, q, \unitint} \qquad & &\forall f \in W_L^{s, q}(\unitint).
		\end{alignat}  
	\end{lemma}
	\begin{proof}
		Let $f \in C^{\infty}(\bar{\unitint})$ and $1 < q < \infty$. Thanks to \cite[Lemma 3.1 eq. (3.3)]{AinCP19Extension},
		\begin{align}
			\label{eq:proof:hardy-derivative-form}
			(\mathcal{H}_L f)^{(n)}(t) = t^{-(n+1)} \int_{0}^{t} u^{n} f^{(n)}(u) \ du = \int_{0}^{1} u^n f^{(n)}(ut) \ du \qquad \forall n \in \mathbb{N}_0.
		\end{align}
		Applying Hardy's inequality \cite[Theorem 327]{Hardy52} gives
		\begin{align*}
			\| (\mathcal{H}_L f)^{(n)}\|_{q, \unitint}^q \leq \int_{\unitint} \left( t^{-1} \int_{0}^{t} |f^{(n)}(u)| \ du \right)^q \ dt  \leq \left( \frac{q}{q-1} \right)^q \|f^{(n)}\|_{q, \unitint}^q.
		\end{align*}
		Consequently, $\mathcal{H}_L$ is a bounded map of $W^{n, q}(\unitint)$ into $W^{n, q}(\unitint)$ for all $n \in \mathbb{N}_0$. \Cref{eq:hardy-average-bound} now follows from interpolation.
		
		Now let $f \in W_L^{s, q}(\unitint)$. Identity \cref{eq:proof:hardy-derivative-form} and inequality \cref{eq:hardy-average-bound} show that $\mathcal{H}_L f \in  W^{s, q}(\unitint)$ and $(\mathcal{H}_L f)^{(i)}(0) = 0 $ for  $0 \leq i < s - \frac{1}{q}$. Consequently, in the case $s - 1/q \notin \mathbb{Z}$, $\mathcal{H}_L f \in  W_L^{s, q}(\unitint)$. For $s-1/q \in \mathbb{Z}$, we set $n = \lfloor s \rfloor$ and apply Hardy's inequality \cite[Theorem 327]{Hardy52} once again:
		\begin{align*}
			\| \tau^{-s} (\mathcal{H}_L f)^{(n)} \|_{q, \unitint}^q \leq \int_{\unitint} \left( t^{-1} \int_{0}^{t} u^{-s} |f^{(n)}(u)| \ ds \right)^q \ dt \leq \left( \frac{q}{q-1} \right)^q \|\tau^{-s} f^{(n)}\|_{q, \unitint}^q.
		\end{align*} 
		\Cref{eq:hardy-average-bound} then shows that $\mathcal{H}_L$ is a bounded map of $W^{s, q}_L(\unitint)$ into $W^{s, q}_L(\unitint)$ for all $1 < q < \infty$, which completes the proof.
	\end{proof}
	
	\begin{corollary}
		\label{cor:weighted-wsp-bounds}
		Let $1 < q < \infty$. For real numbers $s > 0$, there holds:
		\begin{alignat}{2}
			\label{eq:fovert-bound}
			\| \tau^{-1} f\|_{s, q, \unitint} &\lesssim_{s, q} \|f\|_{s+1, q, \unitint} \qquad & &\forall f \in W^{s+1, q}(\unitint) \cap W^{1, q}_L(\unitint), \\
			\label{eq:fovert-left-bound}
			\leftnorm{\tau^{-1} f}{s, q, \unitint} &\lesssim_{s, q} \leftnorm{f}{s+1, q, \unitint} \qquad & & \forall f \in W^{s+1, q}_{L}(\unitint).
		\end{alignat}
		Additionally, for all real $0 < \beta < 1$, there holds
		\begin{alignat}{2}
			\label{eq:fovertkp12-bound}
			\| \tau^{-\beta} f\|_{q, \unitint} &\lesssim_{\beta, q} \leftnorm{f}{\beta, q, \unitint} \qquad & &\forall f \in W^{\beta, q}_L(\unitint).  
		\end{alignat}
	\end{corollary}
	\begin{proof}
		Let $s > 0$ and $1 < q < \infty$. \Cref{eq:fovert-bound} follows from the identity $t^{-1} f(t) = (\mathcal{H}_L f')(t)$ for $f \in W^{s+1, q}(\unitint) \cap W^{1, q}_L(\unitint)$ and \cref{eq:hardy-average-bound}. Similarly, \cref{eq:fovert-left-bound} follows from the same identity and \cref{eq:hardy-average-left-bound}.
		
		Now let $f \in W_L^{\beta, q}(\unitint)$, $0 < \beta < 1$. When $\beta q = 1$, \cref{eq:fovertkp12-bound} follows from the definition of the norm, so suppose that $\beta q  \neq 1$. \Cref{eq:fovertkp12-bound} is implicit in \cite[Theorem 1.4.4.4]{Grisvard85}, but we provide that proof here for completeness. 
		
		(1) We first show that
		\begin{align}
			\label{eq:proof:half space weighted}
			\| \tau^{-\beta} g \|_{q, \mathbb{R}_+} \lesssim_q \|g\|_{\beta, q, \mathbb{R}_+} \qquad \forall g \in W_0^{\beta, q}(\mathbb{R}_+).
		\end{align}
		By density, it suffices to consider $g \in C_c^{\infty}( \mathbb{R}_+ )$. 
		(a) Let $\beta q < 1$. Thanks to the identity
		\begin{align}
			\label{eq:proof:weird-ibp-id}
			\int_{x}^{\infty} y^{-1} g(y) \ dy 
			&= \int_{x}^{\infty} y^{-2} \int_{0}^{y} g(t) \ dt - x^{-1} \int_{0}^{x} g(t) \ dt,
		\end{align}
		which follows from integration by parts, we have
		\begin{align}
			\label{eq:proof:hardy w def}
			g(x) = -w(x) + \int_{x}^{\infty} y^{-1} w(y) \ dy, \qquad \text{where} \quad w(x) = x^{-1} \int_{0}^{x} [g(t) -g(x)] \ dt.
		\end{align}
		Using H\"{o}lder's inequality, we obtain
		\begin{align*}
			\| \tau^{-\beta} w \|_{q, \mathbb{R}_+}^q 
			&\leq \int_{0}^{\infty} x^{-\beta q - 1} \int_{0}^{x} |g(t) - g(x)|^{q} \ dt \ dt 
			\leq \int_{0}^{\infty} \int_{0}^{\infty} \frac{|g(t) - g(x)|^q}{|x-t|^{1 + \beta q}} \ dt \ dx.
		\end{align*}
		Consequently, $	\|\tau^{-\beta} w \|_{q, \mathbb{R}_+} \leq \|g\|_{\beta, q, \mathbb{R}_+}$. Hardy's inequality \cite[Theorem 330]{Hardy52} then gives
		\begin{align*}
			\int_{0}^{\infty} \left| x^{-\beta} \int_{x}^{\infty} y^{-1} w(y) \ dy \right|^q \ dx \lesssim_{q} \int_{0}^{\infty} x^{-\beta q} |w(x)|^q \ dx \leq \|g\|_{\beta, q, \mathbb{R}_+}^q.
		\end{align*}
		Consequently, \cref{eq:proof:half space weighted} holds.
		
		(b) Now assume that $\beta q > 1$. The identity $g(x) = -w(x) - \int_{0}^{x} y^{-1} w(y) \ dy$, where $w$ is defined in \cref{eq:proof:hardy w def}, may be shown similarly to the identity \cref{eq:proof:weird-ibp-id}. Applying Hardy's inequality \cite[Theorem 330]{Hardy52} once again gives
		\begin{align*}
			\int_{0}^{\infty} \left| x^{-\beta} \int_{0}^{x} y^{-1} w(y) \ dy \right|^q \ dx \lesssim_{q} \int_{0}^{\infty} x^{\beta q} |w(x)|^q \ dx \leq \|g\|_{1-\frac{1}{q}, q, \mathbb{R}_+}^q,
		\end{align*}
		and so \cref{eq:proof:half space weighted} holds.
		
		(2) Now let $f \in W_L^{\beta, q}(\unitint)$, $1 < q < \infty$, $\beta q \neq 1$. Let $\tilde{f}$ denote an extension of $f$ to $\mathbb{R}_{+}$ satisfying $\tilde{f}|_{\unitint} = f$ and $\|\tilde{f}\|_{\beta, q, \mathbb{R}_+ } \lesssim_{\beta, q} \|f\|_{\beta, \unitint}$. Many extensions are possible. For example, let $F$ denote the extension of $f$ on $(1/2, 1)$ to all of $\mathbb{R}$ using the linear extension operator of Stein \cite[Chapter 3]{Stein70} and take $\tilde{f} = F$ on $[1, \infty)$. Applying \cref{eq:proof:half space weighted} gives
		\begin{align*}
			\| \tau^{-\beta} f\|_{q, \unitint} \leq \| \tau^{-\beta} \tilde{f} \|_{q, \mathbb{R}_+} \lesssim_q \|\tilde{f}\|_{\beta, q, \mathbb{R}_+} \lesssim_{\beta, q} \|f\|_{\beta, q, \unitint},
		\end{align*}
		which completes the proof.
	\end{proof}

	\begin{lemma}
		\label{lem:left-extension-operator}
		There exists a linear operator $	\mathcal{F}_L : \bigcup_{ \substack{ s \geq 0 \\ 1 < q < \infty } } W_L^{s, q}(\unitint) \to L^{1}(\mathbb{R})$
		satisfying
		\begin{align*}
			\mathcal{F}_L(f) |_{\unitint} = f, \quad \mathcal{F}_L(f) |_{\mathbb{R}_-} = 0, \quad \text{and} \quad \|\mathcal{F}_L(f)\|_{s, q, \mathbb{R}} \lesssim_{s, q} \leftnorm{f}{s, q, \unitint} \qquad \forall f \in W_L^{s, q}(\unitint).
		\end{align*}
		Moreover, the space $C^{\infty}_c((0, 1])$ is dense in $W_L^{s, q}(\unitint)$ for all $s \geq 0$ and $1 < q < \infty$. 
	\end{lemma}
	\begin{proof}
		Let $1 < q <\infty$, $k \in \mathbb{N}_0$, $\beta \in [0, 1)$, $s = k + \beta$, and $f \in W_L^{s, q}(\unitint)$. Let $\tilde{f} \in W^{s, q}(\mathbb{R}_+)$ denote the extension of $f$ to $(0, \infty)$ in the proof of \cref{cor:weighted-wsp-bounds}. We then define $\mathcal{F}_L(f)$ by $\mathcal{F}_L(f) = 0$ on $\mathbb{R}_{-}$ and $\mathcal{F}_L(f) = \tilde{f}$ on $\mathbb{R}_{+}$. Clearly $\mathcal{F}_L$ is a linear operator and if $\beta > 0$, there holds
		\begin{align*}
			\|\mathcal{F}_L(f)\|_{s, q, \mathbb{R}}^q \lesssim_{s, q} 	\|\tilde{f}\|_{s, q, \mathbb{R}_{+}}^q +  \int_{0}^{1} \int_{-\infty}^0 \frac{|f^{(k)}(t)|^q}{ |t-u|^{\beta q+1} } \ du \ dt &\lesssim_{s, q}  \|f\|_{s, q, \unitint}^q + \|\tau^{-\beta} f^{(k)}\|_{q, \unitint}^q \\
			&\lesssim_{s, q} \leftnormsup{f}{s, q, I}{q}
		\end{align*}
		by \cref{eq:fovertkp12-bound}. The case $\beta=0$ follows analogously. Thus, the zero extension of $\tilde{f}$ to all of $\mathbb{R}$ is bounded. By \cite[Theorem 1.4.2.2]{Grisvard85}, there exists a sequence $\{ \tilde{f}_n \} \subset C^{\infty}_c(\mathbb{R}_+)$ such that $\tilde{f}_n \to \tilde{f}$ strongly in $W_L^{s, q}(\mathbb{R}_+)$. Consequently, $f_n := \tilde{f}_n$ on $\unitint$ satisfies $f_n \in C^{\infty}_c((0, 1])$ and $f_n$ converges strongly to $f$ in $W_L^{s, q}(\unitint)$. Thus,  $C^{\infty}_c((0, 1])$ is dense in $W_L^{s, q}(\unitint)$.
	\end{proof}
	
	\begin{lemma}
		\label{lem:left-zz-interpolation}
		For $n \in \mathbb{N}$, $1 < q < \infty$, and $0 < \beta < 1$ with $\beta q \neq 1$ if $q \neq 2$, there holds
		\begin{align}
			\label{eq:left-space-interpolation}
			W_L^{n+\beta-\frac{1}{q}, q}(\unitint) &= [ W_L^{n-\frac{1}{q}, q}(\unitint), W_{L}^{n+1-\frac{1}{q}, q}(\unitint) ]_{\beta, q},
		\end{align}
		with equivalent norms, where brackets indicate the real method of interpolation \cite{BerLof76}.
	\end{lemma}
	\begin{proof}
		Since $C^{\infty}_c((0, 1])$ is dense in both $	W_L^{n+\beta-\frac{1}{q}, q}(\unitint)$ and the interpolation space $[ W_L^{n-\frac{1}{q}, q}(\unitint), W_{L}^{n+1-\frac{1}{q}, q}(\unitint) ]_{\beta, q}$ by \cref{lem:left-extension-operator} and \cite[Theorem 3.4.2]{BerLof76}, it suffices to show that the $W_L^{n+\beta-\frac{1}{q}, q}(\unitint)$ norm is equivalent to the interpolation norm. This equivalence follows from exactly the same arguments as in the proof of \cite[Theorem 14.2.3]{Brenner08} replacing the operator ``$E_S$" and ``$E_G$" with $\mathcal{F}_L$ from \cref{lem:left-extension-operator} and the spaces ``$W^{k}_p(\Omega)$" and ``$W^{k+1}_p(\Omega)$" with $W_L^{n - \frac{1}{q}, q}(\unitint)$ and $W_L^{n + 1 - \frac{1}{q}, q}(\unitint)$. 
	\end{proof}
	
\bibliographystyle{siamplain}
\bibliography{references}

\end{document}